\documentclass[a4paper,10pt]{article}
\usepackage[utf8x]{inputenc}
\pdfoutput=1

\usepackage{nccfoots}
\usepackage{times}
\usepackage{geometry}
\usepackage[T1]{fontenc}
\usepackage{color}
\usepackage{amsmath}
\usepackage{amssymb}
\usepackage{array}
\usepackage{amsthm}
\usepackage{graphicx}
\usepackage{hyperref}
\usepackage{setspace}
\usepackage{stmaryrd}

\theoremstyle{plain}
\newtheorem{thm}{Theorem}[section]
\newtheorem{prop}[thm]{Proposition}

\newtheorem{lem}[thm]{Lemma}

\theoremstyle{definition}

\theoremstyle{remark}
\newtheorem{rem}[thm]{Remark}

\theoremstyle{plain}

\newcommand{\R}{\mathbb{R}}

\newcommand{\CP}{\mathbb{C}P}

\newcommand{\dimn}{\mathrm{dim}}

\newcommand{\supp}{\mathrm{supp}}

\newcommand{\scal}{\mathrm{scal}}
\newcommand{\ric}{\mathrm{Ric}}
\newcommand{\trace}{\mathrm{tr}}

\newcommand{\dv}{\text{ }dV}

\newcommand{\spectrum}{\mathrm{spec}}

\renewcommand{\title}[1]{{\bfseries #1}\par}
\renewcommand{\author}[1]{\medskip{#1}\par\smallskip}
\newcommand{\affiliation}[1]{{\itshape #1}\par}
\newcommand{\email}[1]{E-mail:~\texttt{#1}\par}

\numberwithin{equation}{section}

\begin{document}
\begin{center}
\title{\LARGE Stable and unstable Einstein warped products}
\vspace{3mm}
\author{\Large Klaus Kröncke}
\vspace{3mm}
\affiliation{Universität Hamburg, Fachbereich Mathematik\\Bundesstraße 55\\20146 Hamburg, Germany}

\email{klaus.kroencke@uni-hamburg.de} 
\end{center}
\vspace{2mm}
\begin{abstract}In this article, we systematically investigate the stability properties of certain warped product Einstein manifolds. We characterize stability of these metrics in terms of an eigenvalue condition of the Einstein operator on the base manifold. In particular, we prove that all complete manifolds carrying imaginary Killing spinors are strictly stable. Moreover, we show that Ricci-flat and hyperbolic cones over K\"ahler-Einstein Fano manifolds and over nonnegatively curved Einstein manifolds are stable if the cone has dimension $n\geq10$.
\end{abstract}

\section{Introduction}
Let $(M,g)$ be an Einstein manifold. The Einstein operator $\Delta_E$ acting on $C^{\infty}(S^2M)$ is defined by $\Delta_E=\nabla^*\nabla-2\mathring{R}$, where $S^2M$ is the bundle of symmetric $2$-tensors and $\mathring{R}$ is a zero-order operator, defined by $\mathring{R}h_{ij}=R_{iklj}h^{kl}$. We call an Einstein manifold strictly stable if there exists a constant $C>0$ such that
\begin{align}\label{strictlystable}
\int_{M}\langle \Delta_E h,h\rangle \dv\geq C \left\|h\right\|_{L^2}^2\qquad 
\end{align}
for all compactly supported $h\in C^{\infty}(S^2M)$ satisfying $\int_M\trace h\dv=0$ and $\delta h=0$ where $\delta h$ is the divergence of $h$. We call $(M,g)$ stable, if \eqref{strictlystable} holds with $C=0$ and unstable, if it is not stable.

To give some examples, we mention that the round sphere, the hyperbolic space and their quotients are strictly stable. The flat euclidean space and $\CP^n$ are stable but not strictly stable. A product of positive Einstein metrics is unstable. An open problem in this context is the question, whether there exists an unstable compact Einstein metric of nonpositive scalar curvature. In the zero scalar curvature case, this is also known as the positive mass problem for compact Ricci-flat manifolds \cite{Dai07}. In the noncompact case, unstable Ricci-flat metrics and unstable Einstein manifolds of negative scalar curvature are known  \cite{GPY82,War06,HHS14}.

\Footnotetext{}{2010 \emph{Mathematics Subject Classification.} 58J05,53C25,53C21.}
\Footnotetext{}{\emph{Key words and phrases.} Einstein metrics, linear stability, Ricci-flat cones, hyperbolic cones.}

The above stability condition appears in various settings.
At first, the Einstein operator appears in the second variational formula of the Einstein-Hilbert action at compact Einstein metrics \cite{Bes08}. In this context, an Einstein manifold is called stable resp.\ strictly stable, if the above conditions hold for $TT$-tensors, i.e. trace-free and divergence-free symmetric $2$-tensors. Secondly, stability also plays a role in string theory \cite{GPY82,GHP03} and mathematical general relativity \cite{AMo11}.

The Einstein operator also appears in the second variation of Perelman's entropies \cite{CHI04,Zhu11} and therefore influences the local behaviour of the Ricci flow. It has been shown that compact Einstein metrics are dynamically stable under the Ricci flow, if they are stable in the above sense and if all elements in the kernel of the Einstein operator are integrable \cite{GIK02,Ses06,Kro15b}. It is also possible to get rid of the integrability condition in this context \cite{Kro13,HM14}.

In the noncompact case, dynamical stability properties of Einstein manifolds are much less understood but results have been obtained for certain Einstein metrics \cite{SSS08,Bam15} and a nonnegative lower bound of the Einstein operator also plays an important role there.

Our main motivation was to construct new examples for stable and unstable Einstein metrics and the warped product construction is the simplest method where stability properties can be investigated.
The warped product models we consider are as follows: If $g$ is a Ricci-flat metric, $\tilde{g}=dr^2+e^{2r}g$ is an Einstein metric with scalar curvature $-n(n+1)$. If $g$ is a positive Einstein metric which scaled such that $\scal_g=n(n-1)$, the metric $\tilde{g}=dr^2+r^2g$ is Ricci-flat, whereas $\tilde{g}=dr^2+\sinh^2(r)g$ is Einstein with scalar curvature $-n(n+1)$. The main theorems of this paper are stability theorems about these models. Ricci-flat cones have already been considered in \cite{HHS14} and we partly build up on that work.

\begin{thm}\label{thmexponentialcone}
Let $(M,g)$ be a Ricci-flat metric. Then the Einstein manifold
\begin{align*}
(\widetilde{M},\tilde{g})=(\R\times M,dr^2+e^{2r}g)
\end{align*}
is stable if and only if $(M,g)$ is stable. In this case, $(\widetilde{M},\tilde{g})$ is strictly stable.
\end{thm}
\begin{thm}\label{thmricciflatcone}
Let $(M^n,g)$ be a compact Einstein metric of scalar curvature $n(n-1)$. Then the Ricci-flat cone
\begin{align*}
(\widetilde{M},\tilde{g})=(\R_+\times M,dr^2+r^2g)
\end{align*}
is stable if and only if the smallest eigenvalue of the Einstein operator on $M$ restricted to $TT$-tensors satisfies the bound $\lambda\geq -\frac{(n-1)^2}{4}$.
\end{thm}
\begin{thm}\label{thmhyperboliccone}
Let $(M^n,g)$ be a compact Einstein metric of scalar curvature $n(n-1)$. Then the hyperbolic cone
\begin{align*}
(\widetilde{M},\tilde{g})=(\R_+\times M,dr^2+\sinh^2(r)g)
\end{align*}
is stable if and only if the smallest eigenvalue of the Einstein operator on $M$ restricted to $TT$-tensors satisfies the bound $\lambda\geq -\frac{(n-1)^2}{4}$. In this case, $(\widetilde{M},\tilde{g})$ is strictly stable.
\end{thm}
Interestingly, the lower bound $-\frac{(n-1)^2}{4}$ also appears in physics as a stability condition for Freud-Rubin compactifications in AdS-CFT correspondence and for higher-dimensional black holes, which is known as the Breitenlohner-Freedman bound  (c.f.\ \cite{GHP03} and references therein).

In the proofs of these theorems, one direction is much easier to show than the other one. If the condition on the Einstein operator of the base manifold is not satisfied, we are able to construct a test section $\tilde{h}\in C^{\infty}_{cs}(S^2\widetilde{M})$ violating the stability condition. This direction was proven in \cite{HHS14} in the case of Ricci-flat cones and the converse direction was proven in the case of the base manifold $\CP^2$. However, to conclude stability of the warped product from the condition on the base manifold in the general case requires to carefully construct a suitable decomposition of the space of compactly supported symmetric $2$-tensors which is preserved by the Einstein operator. Additionally, one has to check the the operator is nonnegative on most subspaces of the decomposition, which requires a lot of tedious calculations.

In view of the dynamical stability and instability results under Ricci flow which were obtained so far, the stricly stable Einstein metrics in Theorem \ref{thmexponentialcone} and Theorem \ref{thmhyperboliccone} may be dynamically stable. On the other hand, Ricci flows coming out of cones have been constructed in \cite{FIK03,GK04,Sie13,SS13} and there is a conjecture by Ilmanen which relates instability of cones and nonuniqueness of the Ricci flow with conical initial data \cite{HHS14}.

This paper is organized as follows. In Section \ref{section2}, we construct a decomposition of the space $C^{\infty}_{cs}(S^2\widetilde{M})$ of a general warped product metric with respect to which the quadratic form $h\mapsto (\tilde{\Delta}_Eh,h)_{L^2}$ has a block diagonal form. In Section \ref{section3}, we use these formulas to prove the main theorems. In Section \ref{section4}, we discuss stability and instability of warped product metrics over certain classes of manifolds.

To finish the introduction, we fix some notation and conventions.
The Riemann curvature tensor is defined by the sign convention such that $R_{ijkl}=g(\nabla_{\partial_i}\nabla_{\partial_j}\partial_k-\nabla_{\partial_j}\nabla_{\partial_i}\partial_k,\partial_l)$. The Ricci curvature and the scalar curvature of a metric $g$ are denoted by $\ric_g,\scal_g$, respectively. The rough Laplacian acting on smooth sections of a vector bundle is $\Delta=\nabla^*\nabla=-g^{ij}\nabla^2_{ij}$. The symmetric tensor product is $h\odot k=h\otimes k+k\otimes h$. The divergence of a symmetric $2$-tensor and of a one-form are given by $\delta h_j=-g^{ij}\nabla_ih_{ij}$ and $\delta\omega=-g^{ij}\nabla_i\omega_j$, respectively. The formal adjoint $\delta^*:C^{\infty}(T^*M)\to C^{\infty}(S^2M)$ is $(\delta^*\omega)_{ij}=\frac{1}{2}(\nabla_i\omega_j+\nabla_j\omega_i)$. The space of smooth and compactly supported sections of a vector bundle $E$ is denoted by $C^{\infty}_{cs}(E)$.
\subsection*{Acknowledgements}
The author thanks Stuart Hall for helpful comments on the manuscript. Financial support is gratefully acknowledged to the SFB 1085, funded by the Deutsche Forschungsgemeinschaft.
\section{The Einstein operator on warped products}\label{section2}
In this section, we construct a suitable decomposition of the space of symmetric $2$-tensors on a manifold with a warped product metric. Let $\tilde{g}=dr^2+f(r)^2g$ be a warped product metric on $\widetilde{M}=I\times M$, where $I\subset\R$ is an open interval and let $\tilde{h}\in C_{cs}^{\infty}(S^2\widetilde{M})$.
Let $J\subset I$ and $\Omega\subset M$ be compact subsets with smooth boundaries such that $\supp(\tilde{h})\subset J\times \Omega$. Then $\tilde{h}$ can be decomposed as
\begin{align}\label{tensorexpansion1}
\tilde{h}=\sum_{i,j=1}^{\infty}a_{ij}\varphi_i v_j \cdot dr\otimes dr
+\sum_{i,j=1}^{\infty}b_{ij}\varphi_i \cdot dr\odot f\omega_j+\sum_{i,j=1}^{\infty}c_{ij}\varphi_i\cdot f^2h_j.
\end{align}
Here, $a_{ij},b_{ij},c_{ij}\in\R$ and $\varphi_i, v_i, \omega_i, h_i$ are smooth $L^2$-orthonormal bases of $L^2(J),L^2(\Omega),L^2(T^*\Omega)$ and $L^2(S^2\Omega)$, respectively, which vanish at the boundary. To avoid cumbersome notation, we drop the reference to the projection maps. This composition holds for products of closed manifolds, see \cite[Lemma 3.1]{AMo11} and the argument for compact manifolds with boundary is analogous.
For our purposes, it is more convenient to slightly modify this decomposition to 
\begin{align}\label{tensorexpansion2}
\tilde{h}=\sum_{i,j=1}^{\infty}a_{ij}\varphi_i v_j \tilde{g}
+\sum_{i,j=1}^{\infty}b_{ij}\varphi_i \cdot dr\odot f\omega_j+\sum_{i,j=1}^{\infty}c_{ij}\varphi_i\cdot f^2h_j,
\end{align}
because the Einstein operator preserves the conformal class of the metric, i.e. the first factor of \eqref{tensorexpansion2}, but not the first factor of \eqref{tensorexpansion1}.
Now we want to see how the Einstein operator acts as a quadratic form with respect to this decomposition. Later on, we will further decompose the tensors on $M$ in the last summand.

 From now on, we assume that the base metric and the warped product metric are both Einstein.  Before we start the actual computations, we recall that the nonvanishing Christoffel symbols of the warped product metric are 
\begin{align}
\tilde{\Gamma}_{ij}^k=\Gamma_{ij}^k,\qquad \tilde{\Gamma}_{ij}^0=-f'fg_{ij},\qquad \tilde{\Gamma}_{0i}^k=\tilde{\Gamma}_{i0}^k=\frac{f'}{f}\delta_i^k,
\end{align}
where $0$ denotes the $r$-coordinate and
latin indices denote coordinates on $M$. 
Here and in the following, objects depending on the metric are denoted with tilde if they depend on $\tilde{g}$ and without the tilde if they depend on $g$. 
Recall that the curvature tensors of $\tilde{g}$ and $g$ are related by
\begin{align}
\tilde{R}_{ijkl}=f^2(R_{ijkl}+(f')^2(g_{ik}g_{jl}-g_{il}g_{jk})),\qquad
\tilde{R}_{i00j}=-f''\cdot f g_{ij},
\end{align}
while the other coefficients of $\tilde{R}$ vanish.
\begin{lem}\label{scalarproducts1}Let $\varphi\in C_{cs}^{\infty}(I),v\in C_{cs}^{\infty}(M), \omega\in C_{cs}^{\infty}(T^{*}M),h\in C_{cs}^{\infty}(S^2M)$. 
Furthermore, let
\begin{align}
\tilde{h}_1=\varphi f^2h,\qquad \tilde{h}_2=\varphi dr\odot f\omega,\qquad \tilde{h}_3=\varphi v \tilde{g}.
\end{align}
Then,
\begin{equation}
\begin{aligned}
(\tilde{\Delta}_E\tilde{h}_1,\tilde{h}_1)_{L^2(\tilde{g})}&=
\int_{I}\varphi^2f^{n-2}dr\cdot[({\Delta_E}{h},{h})_{L^2({g})}+2\left\|\trace_gh\right\|^2_{L^2(g)}]\\
&\qquad+\int_{I}(\varphi')^2f^n dr\cdot \left\|h\right\|^2_{L^2(g)},\\
(\tilde{\Delta}_E\tilde{h}_2,\tilde{h}_2)_{L^2(\tilde{g})}&=2\int_{I}f^{n-2}\varphi^2dr\left\|\nabla \omega\right\|_{L^2(g)}^2
 +(2n+6)\int_{I}(f')^2f^{n-2}\varphi^2dr\left\|\omega\right\|_{L^2(g)}^2\\&
 \qquad+2\int_{I}(\varphi')^2f^{n}dr\left\|\omega\right\|_{L^2(g)}^2
 -4\int_{I}f''f^{n-1}\varphi^2dr\left\|\omega\right\|_{L^2(g)}^2,\\
 (\tilde{\Delta}_E\tilde{h}_3,\tilde{h}_3)_{L^2(\tilde{g})}&=
(n+1)\int_{I}(\varphi')^2f^{n}dr\left\|v\right\|^2_{L^2(g)}+
(n+1)\int_{I}\varphi^2f^{n-2}dr\left\|\nabla v\right\|_{L^2(g)}^2\\&\qquad-2\scal_{\tilde{g}}\int_{I}\varphi^2f^{n}dr\left\|v\right\|_{L^2(g)}^2.
\end{aligned}
\end{equation}
\end{lem}
\begin{proof}
Let $\tilde{h}\in C_{cs}^{\infty}(S^2\widetilde{M})$ be of the form $\tilde{h}=\varphi f^2 h$.
The we find that
\begin{align}\label{cov_der_1}
\tilde{\nabla}_i\tilde{h}_{jk}=\varphi f^2\nabla_ih_{jk},\qquad \tilde{\nabla}_0\tilde{h}_{ij}=\varphi 'f^2h_{ij},\qquad \tilde{\nabla}_i\tilde{h}_{j0}=\tilde{\nabla}_i\tilde{h}_{0j}=-\varphi f'f h_{ij},
\end{align}
while the other components vanish. This implies 
\begin{align}\label{cov_der_1.5}
|\tilde{\nabla}\tilde{h}|^2_{\tilde{g}}=\varphi^2\frac{1}{f^2}|\nabla h|^2_g+(\varphi')^2|h|_g+2\frac{\varphi^2}{f^2}(f')^2|h|^2_{g}.
\end{align}
The actions of the curvature tensors are related by
\begin{align}
\langle \mathring{\tilde{R}}\tilde{h},\tilde{h}\rangle_{\tilde{g}}=\frac{\varphi^2}{f^2}[\langle\mathring{R}h,h\rangle_g+(f')^2|h|^2_{g}-(f')^2(\trace_g h)^2].
\end{align}
This yields
\begin{equation}\begin{split}
(\tilde{\Delta}_E\tilde{h},\tilde{h})_{L^2(\tilde{g})}=&
\int_{I}\varphi^2f^{n-2}dr\cdot[({\Delta_E}{h},{h})_{L^2({g})}+2\left\|\trace_gh\right\|^2_{L^2(g)}]\\
&+\int_{I}(\varphi')^2f^n dr\cdot \left\|h\right\|^2_{L^2(g)}.
\end{split}
\end{equation}
Next, let $\tilde{h}\in C_{cs}^{\infty}(S^2\widetilde{M})$ be $\tilde{h}=\varphi dr\odot f\omega$.
Straightforward calculations show that
\begin{equation}
\begin{aligned}\label{cov_der_2}
\tilde{\nabla}_i\tilde{h}_{jk}&=\varphi f'f^2(\omega_jg_{ik}+\omega_kg_{ij}),
&\tilde{\nabla}_i\tilde{h}_{j0}&=\tilde{\nabla}_i\tilde{h}_{0j}=f\nabla_i\omega_j\varphi,\\
\tilde{\nabla}_i\tilde{h}_{00}&=-2f'\omega_i\varphi,
&\tilde{\nabla}_0\tilde{h}_{i0}&=\tilde{\nabla}_0\tilde{h}_{0i}=f\varphi'\omega_i,
\end{aligned}
\end{equation}
while the other components vanish. This implies
  \begin{align*}
 |\tilde{\nabla}\tilde{h}|_{\tilde{g}}^2&
=2f^{-2}\varphi^2|\nabla \omega|_{g}^2+[2(n+1)+4](f')^2f^{-2}\varphi^2|\omega|_g+2(\varphi')^2|\omega|^2_g.
 \end{align*}
Moreover,
\begin{align}
\langle \mathring{\tilde{R}}\tilde{h},\tilde{h}\rangle_{\tilde{g}}=2f''f^{-1}\varphi^2|\omega|_{g}^2.
\end{align}
 Integration yields
\begin{equation}
\begin{aligned}
(\tilde{\Delta}_E\tilde{h},\tilde{h})_{L^2(\tilde{g})}&=2\int_{I}f^{n-2}\varphi^2dr\left\|\nabla \omega\right\|_{L^2(g)}^2
 +(2n+6)\int_{I}(f')^2f^{n-2}\varphi^2dr\left\|\omega\right\|_{L^2(g)}^2\\& \qquad
 +2\int_{I}(\varphi')^2f^{n}dr\left\|\omega\right\|_{L^2(g)}^2
 -4\int_{I}f''f^{n-1}\varphi^2dr\left\|\omega\right\|_{L^2(g)}^2.
\end{aligned}
\end{equation}
Finally, let $\tilde{h}=\varphi v\cdot \tilde{g}$. Then,
\begin{align}
|\tilde{\nabla}\tilde{h}|^2_{\tilde{g}}=(n+1)[(\varphi')^2v^2+f^{-2}\varphi^2|\nabla v|^2_g],
\end{align}
while
\begin{align}
\langle \mathring{\tilde{R}}\tilde{h},\tilde{h}\rangle_{\tilde{g}}=\varphi^2 v^2\scal_{\tilde{g}}.
\end{align}
Thus,
\begin{equation}
\begin{aligned}
(\tilde{\Delta}_E\tilde{h},\tilde{h})_{L^2(\tilde{g})}&=
(n+1)\int_{I}(\varphi')^2f^{n}dr\left\|v\right\|^2_{L^2(g)}+
(n+1)\int_{I}\varphi^2f^{n-2}dr\left\|\nabla v\right\|_{L^2(g)}^2\\&\qquad-2\scal_{\tilde{g}}\int_{I}\varphi^2f^{n}dr\left\|v\right\|_{L^2(g)}^2.
\end{aligned}
\qedhere
\end{equation}
\end{proof}
\begin{lem}\label{mixedscalarproducts1}
Let $\varphi,\psi,\chi\in C_{cs}^{\infty}(I),v\in C_{cs}^{\infty}(M), \omega\in C_{cs}^{\infty}(T^*M),h\in C_{cs}^{\infty}(S^2M)$. 
Furthermore, let
\begin{align}
\tilde{h}_1=\varphi f^2h,\qquad \tilde{h}_2=\psi dr\odot f\omega,\qquad \tilde{h}_3=\chi v \tilde{g}.
\end{align}
Then,
\begin{equation}
\begin{aligned}
(\tilde{\Delta}_E\tilde{h}_1,\tilde{h}_2)_{L^2(\tilde{g})}&=-4
\int_{I}\varphi\psi f'f^{n-2}dr\cdot(\delta h,\omega)_{L^2({g})},\\
(\tilde{\Delta}_E\tilde{h}_1,\tilde{h}_3)_{L^2(\tilde{g})}&=
\int_{I}\varphi\psi f'f^{n-2}dr\cdot[(\nabla \trace_gh,\nabla v)_{L^2({g})}-2\frac{\scal_{\tilde{g}}}{n+1}( \trace_gh,v)_{L^2({g})}]\\
&\qquad+\int_{I}\varphi'\psi' f^{n}dr\cdot(\trace_gh, v)_{L^2({g})},\\
(\tilde{\Delta}_E\tilde{h}_2,\tilde{h}_3)_{L^2(\tilde{g})}&=0.
\end{aligned}
\end{equation}
\end{lem}
\begin{proof}
Let $\tilde{h}_1,\tilde{h}_2$ be as above. Using \eqref{cov_der_1} and \eqref{cov_der_2}, we see that
\begin{align*}
\langle\tilde{\nabla}\tilde{h}_1,\tilde{\nabla}\tilde{h}_2\rangle_{\tilde{g}}=
-2f^{-2}f'\varphi\psi[\langle\delta h,\omega\rangle_g+\langle h,\delta^*\omega\rangle_g].
\end{align*}
and the curvature term vanishes. The first formula follows from integration over $\widetilde{M}$ and integration by parts on $M$.
For the second formula, we compute
\begin{align}
\langle\tilde{\nabla}\tilde{h}_1,\tilde{\nabla}\tilde{h}_3\rangle_{\tilde{g}}=
\varphi\psi f^{-2}\langle\nabla\trace_g h,\nabla v\rangle_g+\varphi'\psi' v\trace_gh,
\end{align}
where we used \eqref{cov_der_2} again and
\begin{align}
\langle\mathring{\tilde{R}}\tilde{h}_3,\tilde{h}_1\rangle_{\tilde{g}}=\frac{\scal_{\tilde{g}}}{n+1}\varphi\chi f^{-2}v\trace_gh.
\end{align}
The third formula follows from the fact that $\tilde{\Delta}_E(\psi v\tilde{g})=(\tilde{\Delta}_g(\psi v) -2\frac{\scal_{\tilde{g}}}{n+1}\psi v)\tilde{g}$, 
and that $\trace_{\tilde{g}}\tilde{h}_2=0$.
\end{proof}
Now we are going to refine the above decomposition. Recall that on any compact Einstein manifold $(M^n,g)$ except the sphere, we have
\begin{align}\label{bessesplitting}
C^{\infty}(S^2M)=C^{\infty}(M)\cdot g\oplus\delta^*(C^{\infty}(T^*M))\oplus \trace_g^{-1}(0)\cap\delta_g^{-1}(0),
\end{align}
see \cite[Lemma 4.57]{Bes08}. Recall that $C^{\infty}(T^*M)$ splits orthogonally to $d(C^{\infty}(M))$ and $C^{\infty}(T^*M)\cap\delta^{-1}(0)$ and that this $L^2$-orthogonality is preserved by the map $\delta^{*}:C^{\infty}(T^*M)\to C^{\infty}(S^2M)$.
Therefore, \eqref{bessesplitting} can be more refined to the $L^2$-orghogonal splitting
\begin{align}\label{tensorsplitting}
C^{\infty}(S^2M)=C^{\infty}(M)\cdot g\oplus W_1\oplus W_2\oplus TT_{g},
\end{align}
where
\begin{equation}
\begin{aligned}
W_1:=&\left\{n\nabla^2 v+\Delta v\cdot g\mid v\in C^{\infty}(M)\right\},\\
W_2:=&\left\{\delta^*\omega\mid \omega\in C^{\infty}(T^*M),\delta\omega=0\right\},
\end{aligned}
\end{equation}
and $TT_{g}=\trace_g^{-1}(0)\cap\delta_g^{-1}(0)$.
 Note that $W_2=\delta^*(C^{\infty}(T^*M)\cap\delta^{-1}(0))=$ by definition but $W_1$ is the projection of $\delta^*(d C^{\infty}(M))$ to $\trace^{-1}(0)$ in order to ensure orthogonality of $W_1$ to all other factors. It is not hard to see that this splitting also holds for compactly supported tensors, if $M$ is noncompact.
 It is preserved by the Einstein operator because we have the formulas
 \begin{equation}\label{commutators}
 \begin{aligned}
 \Delta_E(vg)&=\left(\Delta v-2\frac{\scal}{n}v\right)g,\\
 \Delta_E(\nabla^2v)&=\nabla^2\left(\Delta v-2\frac{\scal}{n}v\right),\\
 \Delta_E(\delta^*\omega)&=\delta^*\left(\Delta_H\omega-2\frac{\scal}{n}\omega\right).
  \end{aligned}
  \end{equation}
where $\Delta_H$ is the Hodge Laplacian, see e.g. \cite[Lemma 4.2]{Kro15}.
 Combining \eqref{tensorexpansion2} and \eqref{tensorsplitting}, we can expand any tensor $\tilde{h}\in C_{sc}^{\infty}(S^2\widetilde{M})$ with support contained in $J\times\Omega$ as
\begin{equation}\label{tensorexpansion3}
\begin{aligned}
\tilde{h}=&
\sum_{i,j=1}^{\infty}a_{ij}\varphi_i f^2 h_j
+\sum_{i,j=1}^{\infty}b^{(1)}_{ij}\varphi_i f^2 \delta^*\omega_j
+\sum_{i,j=1}^{\infty}b^{(2)}_{ij}\varphi_i \cdot dr\odot f\omega_j\\
&+\sum_{i,j=1}^{\infty}c^{(1)}_{ij}\varphi_i v_j \tilde{g}
+\sum_{i,j=1}^{\infty}c^{(2)}_{ij}\varphi_i f^2(n\nabla^2 v_j+\Delta v_j\cdot g)\\
&+\sum_{i,j=1}^{\infty}c^{(3)}_{ij}\varphi_i \cdot dr\odot \nabla v_j
+\sum_{i,j=1}^{\infty}c^{(4)}_{ij}\varphi_i\cdot v_j(f^2g-ndr\otimes dr),
\end{aligned}
\end{equation}
where $a_{ij},b^{(k)}_{ij},c^{(k)}_{ij}\in\R$ and $\varphi_i,v_i,\omega_i, h_i$ are smooth orthonormal bases of $L^2(J),L^2(\Omega),L^2 (W_2(\Omega))$ and $L^2(TT_g(\Omega))$ vanishing at the boundary. Let $\lambda_1\leq\lambda_2\leq\ldots$ be the Dirichlet eigenvalues of the Laplacian on $C^{\infty}(\Omega)$, $\mu_1\leq\mu_2\leq\ldots$ be the dirichlet eigenvalues of the connection Laplacian on $W_2(\Omega)$ and $\kappa_1\leq\kappa_2\leq\ldots$ be the dirichlet eigenvalues of the Einstein operator on $TT_g(\Omega)$. We then may choose the orthonormal bases such that $\Delta v_i=\lambda_i v_i$, $\Delta \omega_i=\mu_i\omega_i$ and $\Delta_E h_i=\kappa_i h_i$.
\begin{rem}
Let us explain why we choose exactly the expansion \eqref{tensorexpansion3}. From 
  \eqref{tensorexpansion2} and \eqref{tensorsplitting}, it is clear that we have $7$ different types of symmetric $2$-tensors on the warped product and each of the types contains either $v_i$, $\omega_i$ or $h_i$.
  Therefore in order to prove positivity of the Einstein operator, we basically have to prove positivity of a $7\times 7$-matrix coming from this decomposition. For this purpose, it is convenient to force as many off-diagonal terms to be zero as possible.

It is intuitively clear that the tensors generated by $v_i$, $\omega_i$ or $h_i$ should not interact with each other. Therefore, the $7\times 7$-matrix splits to a $1\times 1$-block, corresponding to the tensors containing the $h_i$, a $2\times 2$-block, corresponding to the two types of tensors containing the $\omega_i$ and a $4\times 4$-block corresponding to tensors containing the $v_i$. The $1\times 1$ and the $2\times 2$ block are not too hard to analyse.

The $4\times4$-block is most difficult to analyse and therefore we need a clever $L^2$-orthogonal decomposition of the types of tensors containing the $v_i$. From \eqref{tensorexpansion1} and \eqref{bessesplitting}, the first guess would have been to use the expansion
\begin{equation}\label{tensorexpansion4}
\begin{aligned}
\sum_{i,j=1}^{\infty}c^{(1)}_{ij}\varphi_i v_j dr\otimes dr
+\sum_{i,j=1}^{\infty}c^{(2)}_{ij}\varphi_i f^2\nabla^2 v_j
+\sum_{i,j=1}^{\infty}c^{(3)}_{ij}\varphi_i \cdot dr\odot \nabla v_j
+\sum_{i,j=1}^{\infty}c^{(4)}_{ij}\varphi_i\cdot v_jf^2g,
\end{aligned}
\end{equation}
but it turns out to be quite inconvenient because all coefficents of the $4\times4$-block are nonvanishing in this case. Let us now go back to \eqref{tensorexpansion3}.
 There, the fourth sum represents the conformal class of $\tilde{g}$, while the terms in the fifth and seventh sum are sums of two terms in order to ensure that these tensors are trace-free (with respect to $g$ as well as to $\tilde{g}$). The terms in the sixth summand are also trace-free. The Einstein operator preserves the conformal class and the space of trace-free tensors and therefore, $4\times 4$-matrix splits to a $1\times 1$ and a $3\times 3$-block.

Luckily, it turns out in Lemma \ref{mixedscalarproducts2} below that the terms in the fifth summand and the terms in the last summand also do not interact with each other unter the Einstein operator, which means that the $3\times3$-matrix has just $2$ nonzero off-diagonal terms.
\end{rem}

\begin{lem}\label{commutations}Let $(M^n,g)$ be an Einstein manifold. Then for $v\in C^{\infty}_{cs}(M)$ and $\omega\in W_2$, we have
\begin{align}
\delta^*\delta\omega&=\frac{1}{2}\nabla^*\nabla\omega-\frac{\scal}{2n}\omega,\\
(\nabla^2)^*\nabla^2v&=\delta\delta\nabla^2v=\Delta\left(\Delta-\frac{\scal}{n}\right)v.
\end{align}
\end{lem}
\begin{proof}
The first formula follows from \cite[Lemma 4.4]{Kro15}. For the second formula, we compute
\begin{align}
\delta\delta\nabla^2v=\delta (\nabla^*\nabla)\nabla v=\delta \nabla\left(\Delta-\frac{\scal}{n}\right)v=\Delta\left(\Delta-\frac{\scal}{n}\right)v.
\end{align}
Here, the second equality holds due to of the well-known formula $\Delta_H\nabla v=\nabla\Delta v$.
\end{proof}
\begin{lem}\label{norms}Let $\varphi\in C^{\infty}_{cs}(I)$ with support contained in $J$ and $v_i, \omega_i, h_i$ as above.
		Let
		\begin{equation}\label{tensors}
			\begin{aligned}
				\tilde{h}_{1,i}^{(1)}=\varphi f^2\cdot h_i,\quad 
				\tilde{h}_{2,i}^{(1)}=\varphi v_i \tilde{g},\quad 
				\tilde{h}_{3,i}^{(1)}=\varphi f^2\delta^*\omega_i,\quad
				\tilde{h}_{3,i}^{(2)}=\varphi \cdot dr\odot f\omega_i,\\
				\tilde{h}_{4,i}^{(1)}=\varphi f^2(n\nabla^2 v_i+\Delta v_i\cdot g),\quad
				\tilde{h}_{4,i}^{(2)}=\varphi \cdot dr\odot f \nabla v_i,\quad
				\tilde{h}_{4,i}^{(3)}=\varphi\cdot v_i(f^2g-ndr\otimes dr).
			\end{aligned}
		\end{equation}
	The $L^2$-norms of these tensors are
	\begin{equation}
		\begin{aligned}
			\left\|\tilde{h}_{1,i}^{(1)}\right\|_{L^2(\tilde{g})}^2&=\int_J\varphi^2f^ndr,
			&\left\|\tilde{h}_{2,i}^{(1)}\right\|_{L^2(\tilde{g})}^2&=(n+1)\int_J\varphi^2f^ndr,\\
			\left\|\tilde{h}_{3,i}^{(1)}\right\|_{L^2(\tilde{g})}^2&=\frac{1}{2}\left(\mu_i-\frac{\scal_g}{n}\right)\int_J\varphi^2f^ndr, 
			&\left\|\tilde{h}_{3,i}^{(2)}\right\|_{L^2(\tilde{g})}^2&=2\int_J\varphi^2f^ndr,\\
			\left\|\tilde{h}_{4,i}^{(1)}\right\|_{L^2(\tilde{g})}^2&= n\lambda_i[(n-1)\lambda_i-\scal_g]\int_J\varphi^2f^ndr,
			&\left\|\tilde{h}_{4,i}^{(2)}\right\|_{L^2(\tilde{g})}^2&=2\lambda_i\int_J\varphi^2f^ndr,\\
			\left\|\tilde{h}_{4,i}^{(3)}\right\|_{L^2(\tilde{g})}^2&=(n+1)n\int_J\varphi^2f^ndr,
		\end{aligned}
	\end{equation}
	and if $\delta_{il}\cdot\delta_{jm}\cdot\delta_{kn}=0$, $$(\tilde{h}_{i,j}^{(k)},\tilde{h}_{l,m}^{(n)})_{L^2(\tilde{g})}=0.$$
\end{lem}
\begin{lem}\label{scalarproducts2} Let the tensors $\tilde{h}_{i,j}^{(k)}$ be as in \eqref{tensors}.
Then, we have the scalar products
\begin{equation}
\begin{aligned}
(\tilde{\Delta}_E\tilde{h}_{1,i}^{(1)},\tilde{h}_{1,i}^{(1)})_{L^2(\tilde{g})}&=\int_J(\varphi')^2f^n dr
 +\kappa_i\int_J\varphi^2f^{n-2}dr,\\
 (\tilde{\Delta}_E\tilde{h}_{2,i}^{(1)},\tilde{h}_{2,i}^{(1)})_{L^2(\tilde{g})}&=
(n+1)\int_{J}(\varphi')^2f^{n}dr+
(n+1)\lambda_i\int_{J}\varphi^2f^{n-2}dr-2\scal_{\tilde{g}}\int_{J}\varphi^2f^{n}dr,\\
 (\tilde{\Delta}_E\tilde{h}_{3,i}^{(1)},\tilde{h}_{3,i}^{(1)})_{L^2(\tilde{g})}&=
 \frac{1}{2}\left(\mu_i-\frac{\scal_{\tilde{g}}}{n}\right)\int_J(\varphi')^2f^n dr\\
 &\qquad+ \frac{1}{2}\left(\mu_i-\frac{\scal_{\tilde{g}}}{n}\right)^2\int_J\varphi^2f^{n-2}dr,\\
(\tilde{\Delta}_E\tilde{h}_{3,i}^{(2)},\tilde{h}_{3,i}^{(2)})_{L^2(\tilde{g})}&=2\mu_i\int_{J}f^{n-2}\varphi^2dr
 +(2n+6)\int_{J}(f')^2f^{n-2}\varphi^2dr\\&\qquad
 +2\int_{J}(\varphi')^2f^{n}dr -4\int_{J}f''f^{n-1}\varphi^2dr,\\
    (\tilde{\Delta}_E\tilde{h}_{4,i}^{(1)},\tilde{h}_{4,i}^{(1)})_{L^2(\tilde{g})}&= n\lambda_i[(n-1)\lambda_i-\scal_g]\int_J(\varphi')^2f^n dr \\
 &\qquad+n\lambda_i[(n-1)\lambda_i-\scal_g]\left(\lambda_i-2\frac{\scal_g}{n}\right)\int_J\varphi^2f^{n-2}dr,\\
  (\tilde{\Delta}_E\tilde{h}_{4,i}^{(2)},\tilde{h}_{4,i}^{(2)})_{L^2(\tilde{g})}&=(2n+6)\lambda_i\int_{J} \varphi^2(f')^2f^{n-2}dr+2\lambda_i\int_{J} (\varphi')^2f^ndr\\&\qquad+2\lambda_i\left(\lambda_i-\frac{\scal_{g}}{n}\right)\int_{J} \varphi^2f^{n-2}dr
 -4\lambda_i\int_{J}\varphi^2 f''f^{n-1}dr,\\
  (\tilde{\Delta}_E\tilde{h}_{4,i}^{(3)},\tilde{h}_{4,i}^{(3)})_{L^2(\tilde{g})}&= 
  n\left((n+1)\lambda_i-2\frac{\scal_{g}}{n}\right)\int_J \varphi^2 f^{n-2}dr
  +(n+1)n\int_J(\varphi')^2f^ndr\\
  &\qquad+2n^2(n+3)\int_J \varphi^2(f')^2f^{n-2}dr
  -4n^2\int_J\varphi^2 f'' f^{n-1}dr.
 \end{aligned}
 \end{equation}
 Moreover,
 \begin{align}
 (\tilde{\Delta}_E\tilde{h}_{i,j}^{(k)},\tilde{h}_{l,m}^{(n)})_{L^2(\tilde{g})}=0,
 \end{align}
 if $i\neq l$ or $j\neq m$.
\end{lem}
\begin{proof}The first, the second and the fourth formula follow immediately from
Lemma \ref{scalarproducts1}. Moreover, by inserting $h=\delta^*\omega_i$ in Lemma \ref{scalarproducts1}, we get
\begin{align} (\tilde{\Delta}_E\tilde{h}_{3,i}^{(1)},\tilde{h}_{3,i}^{(1)})_{L^2(\tilde{g})}=
\int_{J}\varphi^2f^{n-2}dr\cdot({\Delta_E}{h},{h})_{L^2({g})}
+\int_{J}(\varphi')^2f^n dr\cdot \left\|h\right\|^2_{L^2(g)},
\end{align}
because $\trace_g h=\delta \omega_i=0$. Furthermore,
\begin{align}
\Delta_E \circ\delta^*(\omega_i)=\delta^*\circ \left(\Delta-\frac{\scal}{n}\right)\omega_i,\qquad \delta(\delta^*\omega)=\frac{1}{2} \left(\Delta-\frac{\scal}{n}\right)\omega
\end{align}
by \eqref{commutators} and Lemma \ref{commutations} which proves the third formula. To prove the fifth formula, we first use Lemma \ref{scalarproducts1} for $h=n\nabla^2 v_i+\Delta v_i\cdot g$ to see that
\begin{align} (\tilde{\Delta}_E\tilde{h}_{4,i}^{(1)},\tilde{h}_{4,i}^{(1)})_{L^2(\tilde{g})}=
\int_{J}\varphi^2f^{n-2}dr\cdot({\Delta_E}{h},{h})_{L^2({g})}
+\int_{J}(\varphi')^2f^n dr\cdot \left\|h\right\|^2_{L^2(g)},
\end{align}
because we again have $\trace_g h=0$. Moreover,
\begin{align}
\Delta_E(n\nabla^2 v_i+\Delta v_i\cdot g)=n\nabla^2\left(\Delta-2\frac{\scal}{n}\right) v_i+\Delta \left(\Delta-2\frac{\scal}{n}\right)v_i\cdot g
\end{align}
and 
\begin{align}
 \left\|n\nabla^2 v_i+\Delta v_i\cdot g\right\|^2_{L^2(g)}=n\lambda_i[(n-1)\lambda_i-\scal_g],
\end{align}
 which follows quite immediately from Lemma \ref{commutations}. The sixth formula follows from inserting $\omega=\nabla v_i$ in Lemma \ref{scalarproducts1} and using
\begin{align}
\left\|\nabla^2 v_i\right\|_{L^2(g)}^2=\lambda_i\left(\lambda_i-\frac{\scal_{g}}{n}\right),
\end{align}
which follows from Lemma \ref{commutations} again.
For the last formula, we have to work a bit more. At first, we get from \eqref{cov_der_1.5} that
\begin{align}
|\tilde{\nabla}\varphi f^2 v_i g|^2_{\tilde{g}}=n\varphi^2f^{-2}|\nabla v|_g^2 +n(\varphi')^2 v_i^2+2n\varphi^2 f^{-2}(f')^2v_i^2.
\end{align}
Moerover, straightforward computations, using $\tilde{\nabla}_kdr_l=f'fg_{kl}$ yield
\begin{equation}\begin{split}
\tilde{\nabla}_0(\varphi v_i dr\otimes dr)_{00}&=\varphi'v_i,\qquad
\tilde{\nabla}_k(\varphi v_i dr\otimes dr)_{00}=\varphi\nabla_k v_i,\\
\tilde{\nabla}_k(\varphi v_i dr\otimes dr)_{l0}&=
\tilde{\nabla}_k(\varphi v_i dr\otimes dr)_{0l}=\varphi v_ig_{kl} f'f,\qquad
\end{split}\end{equation}
and therefore,
\begin{equation}\begin{split}
|n\tilde{\nabla}(\varphi v_i dr\otimes dr)|^2_{\tilde{g}}&=n^2(\varphi')^2v_i^2+n^2\varphi^2f^{-2}|\nabla v_i|^2+2n^3\varphi^2f^{-2}(f')^2v_i^2,\\
\langle n\tilde{\nabla}(\varphi v_i dr\otimes dr),\tilde{\nabla}(\varphi f^2v_i g)\rangle_{\tilde{g}}&=-2n^2 f'\cdot f^{-1}v_i^2\varphi^2.
\end{split}\end{equation}
The curvature terms are
\begin{equation}\begin{split}
\langle \mathring{\tilde{R}}(\varphi f^2v_ig),\varphi f^2 v_ig\rangle_{\tilde{g}}&=
\scal_{g}\varphi^2f^{-2}v^2+n\varphi^2 f^{-2}(f')^2v_i^2-n^2\varphi^2(f')^2f^{-2}v_i^2\\
\langle \mathring{\tilde{R}}(\varphi v_idr\otimes dr),\varphi v_idr\otimes dr\rangle_{\tilde{g}}&=0\\
\langle \mathring{\tilde{R}}(\varphi f^2v_ig),n\varphi  v_idr\otimes dr\rangle_{\tilde{g}}&=-n^2\varphi^2f''f^{-1}v_i^2
\end{split}\end{equation}
and the formula follows by adding and integrating. The assertion that
 \begin{align}
 (\tilde{\Delta}_E\tilde{h}_{i,j}^{(k)},\tilde{h}_{l,m}^{(n)})_{L^2(\tilde{g})}=0,
 \end{align}
 if $j\neq m$ follows from polarisation and the fact that the $h_i,\omega_i,v_i$ are eigentensors of the corresponding elliptic operators. That these expressions also vanish for $i\neq l$ and both not equal to $2$ follows from Lemma \ref{scalarproducts1} and Lemma \ref{mixedscalarproducts1} and the properties $\trace_gh_i=0$, $\delta_gh_i=0$ and $\delta_g\omega_i=0$. That the above expression is zero for $i=2$ and $l\neq2$ follows from the fact that $\tilde{\Delta}_E\tilde{h}_{2,j}^{(1)}\in C_{cs}^{\infty}(\widetilde{M})\cdot \tilde{g}$ (since $\tilde{\Delta}_E$ preserves this space) and $\trace_{\tilde{g}}\tilde{h}_{l,m}^{(n)}=0$ for $l\neq 2$.
\end{proof}
\begin{lem}\label{mixedscalarproducts2} Let $\varphi,\psi,\chi\in C^{\infty}_{cs}(I)$ with support contained in $J$ and $v_i, \omega_i, h_i$ as above. Furthermore, let
\begin{equation}
\begin{aligned}
\tilde{h}_{3,i}^{(1)}&=\varphi f^2\delta^*\omega_i,\quad
\tilde{h}_{3,i}^{(2)}=\psi \cdot dr\odot f\omega_i,\quad
\tilde{h}_{4,i}^{(1)}=\varphi f^2(n\nabla^2 v_i+\Delta v_i\cdot g),\\
\tilde{h}_{4,i}^{(2)}&=\psi \cdot dr\odot f \nabla v_i,\quad
\tilde{h}_{4,i}^{(3)}=\chi\cdot v_i(f^2g-ndr\otimes dr).
\end{aligned}
\end{equation}
Then,
\begin{equation}
\begin{aligned}
(\tilde{\Delta}_E\tilde{h}_{3,i}^{(1)},\tilde{h}_{3,i}^{(2)})_{L^2(\tilde{g})}&=
-2\left(\mu_i-\frac{\scal_g}{n}\right)\int_J \varphi\psi f'f^{n-2}dr,\\
(\tilde{\Delta}_E\tilde{h}_{4,i}^{(1)},\tilde{h}_{4,i}^{(2)})_{L^2(\tilde{g})}&=
-4[(n-1)\lambda_i-\scal_g]\lambda_i\int_J \varphi\psi f'f^{n-2}dr,\\
(\tilde{\Delta}_E\tilde{h}_{4,i}^{(1)},\tilde{h}_{4,i}^{(3)})_{L^2(\tilde{g})}&=0,\\
(\tilde{\Delta}_E\tilde{h}_{4,i}^{(2)},\tilde{h}_{4,i}^{(3)})_{L^2(\tilde{g})}&=
4(n+1)\lambda_i\int_J \psi\chi f'f^{n-2}dr.
\end{aligned}
\end{equation}
\end{lem}
\begin{proof}
The first two formula follows from Lemma \ref{mixedscalarproducts1} and Lemma \ref{commutations}. 
To prove the third formula, we first rewrite $\tilde{h}_{4,i}^{(3)}$ as
\begin{align}\label{tensorsplitting2}
\tilde{h}_{4,i}^{(3)}=(n+1)\chi v_i f^2g-n\chi v_i \tilde{g}=:\tilde{h}_1+\tilde{h}_2.
\end{align}
Then,
\begin{align}
(\tilde{\Delta}_E\tilde{h}_{4,i}^{(1)},\tilde{h}_2)_{L^2(\tilde{g})}=0
\end{align}
by Lemma \ref{mixedscalarproducts1}, since $\trace_g\tilde{h}_{4,i}^{(1)}=0$. A polarisation argument in combiniation with the first formula in Lemma \ref{scalarproducts1} yields
\begin{align}
(\tilde{\Delta}_E\tilde{h}_{4,i}^{(1)},\tilde{h}_1)_{L^2(\tilde{g})}=0.
\end{align}
The last formula follows from using \eqref{tensorsplitting2} and Lemma \ref{mixedscalarproducts1}.
\end{proof}
The above lemmas show that the quadratic from $\tilde{h}\mapsto (\tilde{\Delta}_E\tilde{h},\tilde{h})_{L^2(\tilde{g})}$ acting on $C_{cs}^{\infty}(S^2(J\times\Omega))$ is diagonal with respect to the $L^2$-orthogonal decomposition
\begin{align}
L^2(S^2(J\times\Omega))=\bigoplus_{i=1}^{\infty} V_{1,i}\oplus
\bigoplus_{i=0}^{\infty} V_{2,i}\oplus
\bigoplus_{i=1}^{\infty} V_{3,i}\oplus
\bigoplus_{i=0}^{\infty} V_{4,i},
\end{align}
where
\begin{equation}
\begin{aligned}
V_{1,i}=&L^2(J)\cdot f^2h_i,\\
V_{2,i}=&L^2(J)\cdot \tilde{g},\\
V_{3,i}=&L^2(J)\cdot f^2\delta^*\omega_i\oplus L^2(J)\cdot dr\odot f\omega_i,\\
V_{4,i}=&L^2(J)\cdot f^2(n\nabla^2v_i+\Delta v_i\cdot g)\oplus L^2(J)\cdot dr\odot f\nabla v_i
\oplus L^2(J)\cdot v_i(f^2g-ndr\otimes dr).
\end{aligned}
\end{equation}
Thus to prove our main results in the next section, we consider the Einstein operator on each of these subspaces separately.
\section{Proof of the main results}\label{section3}
In this section, we apply the formulas of the previous section to the concrete models of the introduction. Throughout, we may assume that $n\geq 4$. For $n=2,3$, the metric $g$ is of constant curvature so that $\tilde{g}$ is either the flat or the hyperbolic metric on $\R^{n+1}$ for which the result is already known to be true. Note that we even check the stability condition \eqref{strictlystable} for all $h\in C_{cs}^{\infty}(S^2M)$.
\subsection{The metric $dr^2+e^{2r}g$}
 In this section, we prove Theorem \ref{thmexponentialcone}. We start with the following
\begin{lem}\label{infima_1}We have
\begin{align*}
\inf_{\varphi\in C^{\infty}_{cs}(\R)}\frac{\int_{\R} (\varphi')^2e^{nr}dr}{\int_{\R} \varphi^2e^{(n-2)r}dr}=0,\qquad
\inf_{\varphi\in C^{\infty}_{cs}(\R)}\frac{\int_{\R} (\varphi')^2e^{nr}dr}{\int_{\R} \varphi^2e^{nr}dr}=\frac{n^2}{4}.
\end{align*}
\end{lem} 
 \begin{proof}
 By substitution $s=e^r$, the first term transforms as
 \begin{align*}
 \inf_{\varphi\in C^{\infty}_{cs}(\R)}\frac{\int_{\R} (\partial_r\varphi)^2e^{nr}dr}{\int_{\R} \varphi^2e^{(n-2)r}dr}= \inf_{\varphi\in C^{\infty}_{cs}((0,\infty))}\frac{\int_{0}^{\infty} (\partial_s\varphi)^2s^{n+1}ds}{\int_{0}^{\infty} \varphi^2s^{n-3}ds}=:C
 \end{align*}
 and by substituting $s=\epsilon t$, we see that $C=\varepsilon^2\cdot C$ which proves the first equality. Again by substituting $s=e^r$, we have
  \begin{align*}
 \inf_{\varphi\in C^{\infty}_c(\R)}\frac{\int_{\R} (\partial_r\varphi)^2e^{nr}dr}{\int_{\R} \varphi^2e^{nr}dr}= \inf_{\varphi\in C^{\infty}_{cs}((0,\infty))}\frac{\int_{0}^{\infty} (\partial_s\varphi)^2s^{n+1}ds}{\int_{0}^{\infty} \varphi^2s^{n-1}ds}=\frac{n^2}{4}
 \end{align*}
 by the Hardy inequality.
 \end{proof}
  We now study the Einstein operator as a quadratic form on the subspaces $V_{k,i}$ introduced in the last section. Let
 \begin{align}
 \tilde{h}=\varphi e^{2r}h_i\in V_{1,i}.
 \end{align}
 Then
 \begin{align}
 (\tilde{\Delta}_E\tilde{h},\tilde{h})_{L^2(\tilde{g})}&=\int_{\R}(\varphi')^2e^{nr} dr
 +\kappa_i\int_{\R}\varphi^2e^{(n-2)r}dr.
 \end{align}
 By Lemma \ref{infima_1}, this expression is nonnegative for any choice of $\varphi$ if and only if $\kappa_i\geq0$. In this case, the form is even strictly positive, since
 \begin{align}
 (\tilde{\Delta}_E\tilde{h},\tilde{h})_{L^2(\tilde{g})}\geq\int_{\R}(\varphi')^2e^{nr} dr\geq \frac{n^2}{4}\int_{\R}\varphi^2e^{nr} dr=\frac{n^2}{4}\left\|\tilde{h}\right\|_{L^2(\tilde{g})}^2.
 \end{align}
 Now we are going to show that on all other $V_{k,i}$, $k>1$, $\tilde{\Delta}_E$ is always positive definite. Pick
  \begin{align}
 \tilde{h}=\varphi v_i\tilde{g}\in V_{2,i},
 \end{align}
 then
 \begin{equation}\begin{split}
 (\tilde{\Delta}_E\tilde{h},\tilde{h})_{L^2(\tilde{g})}&=
(n+1)\int_{\R}(\varphi')^2e^{nr}dr+
(n+1)\lambda_i\int_{\R}\varphi^2e^{(n-2)r}dr+2n(n+1)\int_{\R}\varphi^2e^{nr}dr\\
&\geq \left(\frac{n^2}{4}+2n\right)(n+1)\int_{\R}\varphi^2e^{nr}dr=\left(\frac{n^2}{4}+2n\right)\left\|\tilde{h}\right\|_{L^2(\tilde{g})}^2,
\end{split}
\end{equation}
since $\lambda_i\geq0$ for all $i\geq0$. Next, let
\begin{align}
 \tilde{h}= \tilde{h}_1+ \tilde{h}_2=\varphi f^2\delta^*\omega_i+
\psi \cdot dr\odot f\omega_i \in V_{3,i}.
\end{align}
Then we have the scalar products
\begin{equation}
\begin{aligned}
 (\tilde{\Delta}_E\tilde{h}_1,\tilde{h}_1)_{L^2(\tilde{g})}&=
 \frac{1}{2}\mu_i\int_{\R}(\varphi')^2e^{nr} dr
 + \frac{1}{2}\mu_i^2\int_{\R}\varphi^2e^{(n-2)r}dr,\\
(\tilde{\Delta}_E\tilde{h}_2,\tilde{h}_2)_{L^2(\tilde{g})}&=2\mu_i\int_{\R}\psi^2e^{(n-2)r}dr
 +(2n+6)\int_{\R}\psi^2e^{nr}dr\\&
 +2\int_{\R}(\psi')^2e^{nr}dr -4\int_{\R}\psi^2e^{nr}dr,\\
 (\tilde{\Delta}_E\tilde{h}_1,\tilde{h}_2)_{L^2(\tilde{g})}&=-2\mu_i\int_{\R}\varphi\psi e^{(n-1)r}dr,
 \end{aligned}
 \end{equation}
 which we can estimate by
\begin{equation}
\begin{aligned}
  (\tilde{\Delta}_E\tilde{h}_1,\tilde{h}_1)_{L^2(\tilde{g})}&\geq 
  \frac{n^2}{8}\mu_i\int_{\R}\varphi^2e^{nr} dr
 + \frac{1}{2}\mu_i^2\int_{\R}\varphi^2e^{(n-2)r}dr,\\
 (\tilde{\Delta}_E\tilde{h}_2,\tilde{h}_2)_{L^2(\tilde{g})}&\geq2\mu_i\int_{\R}\psi^2e^{(n-2)r}dr
 +2(n+1)\int_{\R}\psi^2e^{nr}dr
 +\frac{n^2}{2}\int_{\R}\psi^2e^{nr}dr, \\
 2|(\tilde{\Delta}_E\tilde{h}_1,\tilde{h}_2)_{L^2(\tilde{g})}|& \leq
 \frac{1}{2}\mu_i^2\int_{\R}\varphi^2 e^{(n-2)r}dr+8\int_{\R}\psi^2e^{nr}dr.
 \end{aligned}
 \end{equation}
 Since $n\geq 4$, we obtain
\begin{equation}
\begin{aligned}
(\tilde{\Delta}_E(\tilde{h}_1+\tilde{h}_2),\tilde{h}_1+\tilde{h}_2)_{L^2(\tilde{g})}
&= (\tilde{\Delta}_E\tilde{h}_1,\tilde{h}_1)_{L^2(\tilde{g})}+2(\tilde{\Delta}_E\tilde{h}_1,\tilde{h}_2)_{L^2(\tilde{g})}+(\tilde{\Delta}_E\tilde{h}_2,\tilde{h}_2)_{L^2(\tilde{g})}\\
&\geq \frac{n^2}{4}\frac{\mu_i}{2}\int_{\R}\varphi^2e^{nr} dr +\frac{n^2}{4}2\int_{\R}\psi^2e^{nr}dr \\
&= \frac{n^2}{4}\left(\left\|\tilde{h}_1\right\|_{L^2(\tilde{g})}^2+\left\|\tilde{h}_2\right\|_{L^2(\tilde{g})}^2\right)=\frac{n^2}{4}\left\|\tilde{h}\right\|_{L^2(\tilde{g})}^2,
 \end{aligned}
 \end{equation}
  which shows that $\tilde{\Delta}_E$ is strictly positive on the spaces $V_{3,i}$.
 It remains to consider the spaces $V_{4,i}$. Let
 \begin{align}
  \tilde{h}= \tilde{h}_1+ \tilde{h}_2+\tilde{h}_3=\varphi f^2(n\nabla^2 v_i+\Delta v_i\cdot g)+\psi \cdot dr\odot \nabla v_i+\chi\cdot v_i(f^2g-ndr\otimes dr)\in V_{4,i}.
 \end{align}
Using Lemma \ref{scalarproducts2} and Lemma \ref{mixedscalarproducts2}, have the scalar products 
\begin{equation}
\begin{aligned}
  (\tilde{\Delta}_E\tilde{h}_1,\tilde{h}_1)_{L^2(\tilde{g})}&= n(n-1)\lambda_i^2\int_{\R}(\varphi')^2e^{nr} dr+n(n-1)\lambda_i^3\int_{\R}\varphi^2e^{(n-2)r}dr,\\
  (\tilde{\Delta}_E\tilde{h}_2,\tilde{h}_2)_{L^2(\tilde{g})}&=(2n+6)\lambda_i\int_{\R} \psi^2e^{nr}dr+2\lambda_i\int_{\R} (\psi')^2e^{nr}dr\\&\qquad+2\lambda_i^2\int_{\R} \psi^2e^{(n-2)r}dr
 -4\lambda_i\int_{\R}\psi^2 e^{nr}dr,\\
  (\tilde{\Delta}_E\tilde{h}_3,\tilde{h}_3)_{L^2(\tilde{g})}&= 
  n(n+1)\lambda_i\int_{\R} \varphi^2 e^{(n-2)r}dr
  +(n+1)n\int_{\R}(\varphi')^2e^{nr}dr\\
  &\qquad+2n^2(n+3)\int_{\R} \varphi^2e^{nr}dr
  -4n^2\int_{\R}\varphi^2 e^{nr}dr, \\
  (\tilde{\Delta}_E\tilde{h}_1,\tilde{h}_2)_{L^2(\tilde{g})}&=  -4(n-1)\lambda_i^2\int_{\R}\varphi\psi e^{(n-1)r}dr,\\
    (\tilde{\Delta}_E\tilde{h}_1,\tilde{h}_3)_{L^2(\tilde{g})}&=0,\\
      (\tilde{\Delta}_E\tilde{h}_2,\tilde{h}_3)_{L^2(\tilde{g})}&=
      4(n+1)\lambda_i\int_{\R}\psi\chi e^{(n-1)r}dr.
 \end{aligned}
 \end{equation}
 By Lemma \ref{infima_1}, we have lower estimates
\begin{equation}
\begin{aligned}
 (\tilde{\Delta}_E\tilde{h}_1,\tilde{h}_1)_{L^2(\tilde{g})}&\geq \frac{n^3}{4}(n-1)\lambda_i^2\int_{\R}\varphi^2e^{nr} dr+n(n-1)\lambda_i^3\int_{\R}\varphi^2e^{(n-2)r}dr,\\
 (\tilde{\Delta}_E\tilde{h}_2,\tilde{h}_2)_{L^2(\tilde{g})}&\geq 2(n+1)\lambda_i\int_{\R} \psi^2e^{nr}dr+\frac{n^2}{2}\lambda_i\int_{\R} \psi^2e^{nr}dr+2\lambda_i^2\int_{\R} \psi^2e^{(n-2)r}dr,\\
 (\tilde{\Delta}_E\tilde{h}_3,\tilde{h}_3)_{L^2(\tilde{g})}&\geq 
  n(n+1)\lambda_i\int_{\R} \chi^2 e^{(n-2)r}dr
  +\frac{n^3}{4}(n+1)\int_{\R}\chi^2e^{nr}dr\\ &\qquad+2n^2(n+1)\int_{\R} \chi^2e^{nr}dr,
 \end{aligned}
 \end{equation}
 and by the Young inequality, we get upper estimates
\begin{equation}
\begin{aligned}
2| (\tilde{\Delta}_E\tilde{h}_1,\tilde{h}_2)_{L^2(\tilde{g})}|&\leq
n(n-1)\lambda_i^3\int_{\R}\varphi^2e^{(n-2)r}dr+4\frac{n-1}{n}\lambda_i\int_{\R}\psi^2e^{nr}dr\\
&\qquad+\frac{n^3}{16}(n-1)\lambda_i^2\int_{\R}\varphi^2e^{nr}dr+\frac{64\lambda_i^2}{n^3(n-1)}
\int_{\R}\psi^2e^{(n-2)r}dr,\\
2| (\tilde{\Delta}_E\tilde{h}_2,\tilde{h}_3)_{L^2(\tilde{g})}|&\leq
4\frac{n+1}{n}\lambda_i\int_{\R}\psi^2 e^{nr}dr
+(n+1)n\lambda_i\int_{\R}\chi^2 e^{(n-2)r}dr\\
&\qquad+\frac{2(n+1)}{n^2(n+1)}\int_{\R}\psi^2e^{(n-2)r}dr+2n^2(n+1)\int_{\R}\chi^2e^{nr}dr.
 \end{aligned}
 \end{equation}
 Thus,
\begin{equation}
\begin{aligned}
 (\tilde{\Delta}_E(\tilde{h}_1+\tilde{h}_2+\tilde{h}_3),\tilde{h}_1+\tilde{h}_2+\tilde{h}_3)_{L^2(\tilde{g})}&\geq \frac{3}{4}\frac{n^3}{4}(n-1)\lambda_i^2\int_{\R}\varphi^2e^{nr} dr+\frac{n^2}{2}\lambda_i\int_{\R} \psi^2e^{nr}dr\\& +\frac{n^3}{4}(n+1)\int_{\R}\chi^2e^{nr}dr\\
 &\geq \frac{3}{4}\frac{n^2}{4}\left\|\tilde{h}_1\right\|_{L^2(\tilde{g})}^2+\frac{n^2}{4}\left\|\tilde{h}_2\right\|_{L^2(\tilde{g})}^2+\frac{n^2}{4}\left\|\tilde{h}_3\right\|_{L^2(\tilde{g})}^2\\
 &\geq  \frac{3}{4}\frac{n^2}{4}\left\|\tilde{h}\right\|_{L^2(\tilde{g})}^2,
 \end{aligned}
 \end{equation}
 which shows that $\tilde{\Delta}_E$ is positive on $V_{4,i}$. This proves the theorem.
 \subsection{The Ricci-flat cone}

In this section, we prove Theorem \ref{thmricciflatcone}. We will frequently use the fact that
\begin{align}\label{Hardy}
 \inf_{\varphi\in C^{\infty}_c((0,\infty))}\frac{\int_{0}^{\infty} (\varphi')^2r^{n}dr}{\int_{0}^{\infty} \varphi^2r^{n-2}dr}=\frac{(n-1)^2}{4},
\end{align}
which follows from the Hardy inequality.
The arguments here are analogous to those of the previous subsection. 
Here, we may assume that $\Omega=M$ since $M$ is compact. We also have $\mu_i\geq (n-1)$, $\lambda_1=0$ and $\lambda_i\geq n$ for all $i\geq2$ by Obata's eigenvalue estimate \cite[Theorem 1 and Theorem 2]{Ob62}.
Let
 \begin{align}
 \tilde{h}=\varphi r^2h_i\in V_{1,i}.
 \end{align}
 Then
 \begin{equation}
 \begin{aligned}
 (\tilde{\Delta}_E\tilde{h},\tilde{h})_{L^2(\tilde{g})}&=\int_{0}^{\infty}(\varphi')^2r^n dr
 +\kappa_i\int_{0}^{\infty}\varphi^2r^{n-2}dr
 \geq\left[\frac{(n-1)^2}{4}+\kappa_i\right]\int_{0}^{\infty}\varphi^2r^{n-2}dr .
 \end{aligned}
 \end{equation}
 Since this inequality is optimal by \eqref{Hardy}, the left-hand side is positive for any choice of $\varphi$ if and only if $\kappa_i\geq -\frac{1}{4}(n-1)^2$.
  \begin{rem}
Ricci-flat cones are never strictly stable. By the above, we can choose for any $\epsilon_0>0$ a function $\varphi$ such that
\begin{align}\label{notstrictlypositive}
\frac{ (\tilde{\Delta}_E\tilde{h},\tilde{h})_{L^2(\tilde{g})}}{\left\|\tilde{h}\right\|_{L^2(\tilde{g})}}\leq \left[\frac{(n-1)^2}{4}+\kappa_i+\epsilon_0\right] \frac{\int_{0}^{\infty}\varphi^2r^{n-2}dr }{\int_{0}^{\infty}\varphi^2r^{n}dr}.
\end{align}
Under the transformation $\varphi(r)\to \varphi(\epsilon r) $, \eqref{notstrictlypositive} still holds because the quotient in \eqref{Hardy} is invariant under this transformation. As $\epsilon\to 0$, the right-hand side of \eqref{notstrictlypositive} converges to zero which violates strict stability.
 \end{rem}
 \noindent
   Now pick
  \begin{align}
 \tilde{h}=\varphi v_i\tilde{g}\in V_{2,i},
 \end{align}
 then
 \begin{align}
 (\tilde{\Delta}_E\tilde{h},\tilde{h})_{L^2(\tilde{g})}&=
(n+1)\int_{0}^{\infty}(\varphi')^2r^{n}dr+
(n+1)\lambda_i\int_{0}^{\infty}\varphi^2r^{n-2}dr\geq0
 \end{align}
 by Lemma \ref{scalarproducts2}.
 Next, let
\begin{align}
 \tilde{h}= \tilde{h}_1+ \tilde{h}_2=\varphi f^2\delta^*\omega_i+
\psi \cdot dr\odot f\omega_i \in V_{3,i}.
\end{align}
Then by Lemma \ref{scalarproducts2} and Lemma \ref{mixedscalarproducts2},
\begin{equation}
\begin{aligned}
(\tilde{\Delta}_E\tilde{h}_1,\tilde{h}_1)_{L^2(\tilde{g})}&=
 \frac{1}{2}(\mu_i-(n-1))\int_{0}^{\infty}(\varphi')^2r^{n} dr,\\
 &\qquad+ \frac{1}{2}(\mu_i-(n-1))^2\int_{0}^{\infty}\varphi^2r^{n-2}dr,\\
(\tilde{\Delta}_E\tilde{h}_2,\tilde{h}_2)_{L^2(\tilde{g})}&=2\mu_i\int_{0}^{\infty}\psi^2r^{n-2}dr
 +(2n+6)\int_{0}^{\infty}\psi^2r^{n-2}dr
 +2\int_{0}^{\infty}(\psi')^2r^{n}dr,\\
 (\tilde{\Delta}_E\tilde{h}_1,\tilde{h}_2)_{L^2(\tilde{g})}&=-2(\mu_i-(n-1))\int_{0}^{\infty}\varphi\psi r^{n-2}dr,
 \end{aligned}
 \end{equation}
 and by the Hardy inequality and the Young inequality,
\begin{equation}
\begin{aligned}
  (\tilde{\Delta}_E\tilde{h}_1,\tilde{h}_1)_{L^2(\tilde{g})}&\geq
 \frac{(n-1)^2}{8}(\mu_i-(n-1))\int_{0}^{\infty}\varphi^2r^{n-2} dr\\
 &\qquad+ \frac{1}{2}(\mu_i-(n-1))^2\int_{0}^{\infty}\varphi^2r^{n-2}dr,\\
(\tilde{\Delta}_E\tilde{h}_2,\tilde{h}_2)_{L^2(\tilde{g})}&\geq2\mu_i\int_{0}^{\infty}\psi^2r^{n-2}dr
 +(2n+6)\int_{0}^{\infty}\psi^2r^{n}dr\\& \qquad+\frac{(n-1)^2}{2}\int_{0}^{\infty}\psi^2r^{n-2}dr,\\
 2| (\tilde{\Delta}_E\tilde{h}_1,\tilde{h}_2)_{L^2(\tilde{g})} |&\leq \frac{1}{2}(\mu_i-(n-1))^2\int_{0}^{\infty}\varphi^2r^{n-2}dr+8\int_{0}^{\infty}\psi^2r^{n-2}dr,
 \end{aligned}
 \end{equation}
 so it follows that
 \begin{align}
  (\tilde{\Delta}_E(\tilde{h}_1+\tilde{h}_2),\tilde{h}_1+\tilde{h}_2)_{L^2(\tilde{g})}&\geq 0.
  \end{align}
   It now just remains to consider the spaces $V_{4,i}$. Let
 \begin{align}
  \tilde{h}= \tilde{h}_1+ \tilde{h}_2+\tilde{h}_3=\varphi f^2(n\nabla^2 v_i+\Delta v_i\cdot g)+\psi \cdot dr\odot \nabla v_i+\chi\cdot v_i(f^2g-ndr\otimes dr)\in V_{4,i}.
 \end{align}
We have the scalar products 
\begin{equation}
\begin{aligned}   (\tilde{\Delta}_E\tilde{h}_1,\tilde{h}_1)_{L^2(\tilde{g})}&= (n-1)n\lambda_i(\lambda_i-n)\int_0^{\infty}(\varphi')^2r^n dr \\
 &\qquad+n(n-1)\lambda_i(\lambda_i-n)(\lambda_i-2(n-1))\int_0^{\infty}\varphi^2r^{n-2}dr,\\
  (\tilde{\Delta}_E\tilde{h}_2,\tilde{h}_2)_{L^2(\tilde{g})}&=(2n+6)\lambda_i\int_{0}^{\infty} \varphi^2r^{n-2}dr+2\lambda_i\int_{0}^{\infty} (\varphi')^2r^ndr\\&\qquad+2\lambda_i\left(\lambda_i-(n-1)\right)\int_{0}^{\infty} \varphi^2r^{n-2}dr,
\\
  (\tilde{\Delta}_E\tilde{h}_3,\tilde{h}_3)_{L^2(\tilde{g})}&= 
  n((n+1)\lambda_i-2(n-1))\int_{0}^{\infty} \varphi^2 r^{n-2}dr
  +(n+1)n\int_{0}^{\infty}(\varphi')^2r^ndr\\
  &\qquad+2n^2(n+3)\int_{0}^{\infty} \varphi^2r^{n-2}dr,
 \end{aligned}
 \end{equation}
  and
\begin{equation}
\begin{aligned} 
(\tilde{\Delta}_E\tilde{h}_1,\tilde{h}_2)_{L^2(\tilde{g})}&= 
-4(n-1)\lambda_i(\lambda_i-n)\int_0^{\infty}\varphi\psi r^{n-2}dr,\\
(\tilde{\Delta}_E\tilde{h}_2,\tilde{h}_3)_{L^2(\tilde{g})}&= 
4(n+1)\lambda_i\int_0^{\infty}\psi\chi r^{n-2}dr.
 \end{aligned}
 \end{equation}
  By the Hardy inequality, we have lower estimates
\begin{equation}
\begin{aligned}     (\tilde{\Delta}_E\tilde{h}_1,\tilde{h}_1)_{L^2(\tilde{g})}&\geq (n-1)n\lambda_i(\lambda_i-n)^2\int_0^{\infty}\varphi^2r^{n-2} dr \\
 &\qquad+\frac{1}{4}n(n-1)(n-3)^2\lambda_i(\lambda_i-n)\int_0^{\infty}\varphi^2r^{n-2}dr,\\
  (\tilde{\Delta}_E\tilde{h}_2,\tilde{h}_2)_{L^2(\tilde{g})}&\geq(2n+8+\frac{1}{2}(n-1)^2)\lambda_i\int_0^{\infty} \psi^2r^{n-2}dr+2\lambda_i(\lambda_i-n)\int_0^{\infty} \psi^2r^{n-2}dr,\\
  (\tilde{\Delta}_E\tilde{h}_3,\tilde{h}_3)_{L^2(\tilde{g})}&\geq
  n(n+1)\lambda_i\int_0^{\infty} \varphi^2 r^{n-2}dr
  +\frac{1}{4}(n+1)n(n+3)^2\int_0^{\infty}\varphi^2r^{n-2}dr,
 \end{aligned}
 \end{equation}
and for the off-diagonal terms, we use the Young inequality to show
\begin{equation}
\begin{aligned}
  2|(\tilde{\Delta}_E\tilde{h}_1,\tilde{h}_2)_{L^2(\tilde{g})}|&\leq
  2\alpha (n-1)\lambda_i(\lambda_i-n)\int_0^{\infty}|\varphi\psi| r^{n-2}dr\\&\qquad+
  2(4-\alpha)(n-1)\lambda_i(\lambda_i-n)\int_0^{\infty}|\varphi\psi| r^{n-2}dr\\
  &\leq n(n-1)\lambda_i(\lambda_i-n)^2\int_0^{\infty}\varphi^2r^{n-2}dr
  + \alpha^2\frac{n-1}{n}\lambda_i\int_0^{\infty}\psi^2r^{n-2}dr\\
  &\qquad+\frac{1}{4}n(n-1)(n-3)^2\lambda_i(\lambda_i-n)\int_0^{\infty}\varphi^2r^{n-2}dr\\
  &\qquad+ (4-\alpha)^2\frac{4(n-1)}{n(n-3)^2}\lambda_i(\lambda_i-n)\int_0^{\infty}\psi^2r^{n-2}dr,\\
   2|(\tilde{\Delta}_E\tilde{h}_2,\tilde{h}_3)_{L^2(\tilde{g})}|&\leq 
   16\frac{n+1}{n}\lambda_i \int_0^{\infty}\psi^2r^{n-2}dr
   +n(n+1)\lambda_i\int_0^{\infty}\chi^2r^{n-2}dr,
 \end{aligned}
 \end{equation} 
where $\alpha\in[0,4]$ is some parameter which we choose later. We get
\begin{equation}
\begin{aligned}
(\tilde{\Delta}_E(\tilde{h}_1+\tilde{h}_2+\tilde{h}_3)&,\tilde{h}_1+\tilde{h}_2+\tilde{h}_3)_{L^2(\tilde{g})}\\ &\geq \left[2n+8+\frac{1}{2}(n-1)^2-16\frac{n+1}{n}-\alpha^2\frac{n-1}{n}\right]\lambda_i\int_0^{\infty} \psi^2r^{n-2}dr\\
&\qquad+\left[2-(4-\alpha)^2\frac{4(n-1)}{n(n-3)^2}\right]\lambda_i(\lambda_i-n)\int_0^{\infty} \psi^2r^{n-2}dr.
 \end{aligned}
 \end{equation} 
The right hand side is nonnegative for $n\geq 6$ if we choose $\alpha=2$. If $n=5$, we may choose $\alpha=4-\sqrt{2}$ to get the right sign. In dimension $4$ we are unfortunately not able to choose an appropriate $\alpha$. In this case, we argue as follows: Suppose that $\varphi=\psi=\chi$ and that $\int_{0}^{\infty}\varphi^2r^{n-2}dr=1$. We consider $\tilde{\Delta}_E$ on the space
\begin{align}
 \text{span}\left\{\varphi f^2(n\nabla^2 v_i+\Delta v_i\cdot g),\varphi \cdot dr\odot \nabla v_i,\varphi\cdot v_i(f^2g-ndr\otimes dr)\right\},
\end{align}
and from the above, it can be estimated from below by the quadratic form given by the matrix
\begin{align}
A=\begin{pmatrix} (12\lambda_i-45)\lambda_i(\lambda_i-4) & -12\lambda_i(\lambda_i-4) & 0 \\
-12\lambda_i(\lambda_i-4) & 2\lambda_i^2+\frac{25}{2}\lambda_i & 20\lambda_i \\
0 & 20\lambda_i & 20\lambda_i+245
\end{pmatrix}.
\end{align}
Using $\lambda_i\geq 4$ for $i\geq2$, it can be straightforwardly shown that this matrix is positive definite. In the case $\lambda_1=0$, it is clearly positive semidefinite.
For general $\varphi,\psi,\chi\in C^{\infty}_{cs}((0,\infty))$ with support contained in $[a,b]$, we may use a decomposition into a basis of $L^2([a,b])$ which is  orthnormal with respect to the scalar product $\langle \varphi,\psi\rangle=\int_a^b \varphi\psi r^{n-2} dr$. This shows that $\tilde{\Delta}_E$ is nonnegative on all of $V_{4,i}$.
\subsection{The hyperbolic cone}
Finally, we prove Theorem \ref{thmhyperboliccone}. As in the previous sections, we first need to compute some infima.
\begin{lem}\label{infima_2}We have
\begin{equation}
\begin{aligned}
\inf_{\varphi\in C^{\infty}_{cs}((0,\infty))}\frac{\int_{0}^{\infty} (\varphi')^2\sinh^ndr}{\int_{0}^{\infty} \varphi^2\sinh^{n-2}dr}&=
\inf_{\varphi\in C^{\infty}_{cs}((0,\infty))}\frac{\int_{0}^{\infty} (\varphi')^2\sinh^ndr}{\int_{0}^{\infty} \varphi^2\cosh^2\sinh^{n-2}dr}&=\frac{(n-1)^2}{4}.
\end{aligned}
\end{equation}
\end{lem}
\begin{proof}
By substituting $s=\sinh(r)$, we obtain
\begin{align}
A:=\inf_{\varphi\in C^{\infty}_{cs}((0,\infty))}\frac{\int_{0}^{\infty} (\partial_r\varphi)^2\sinh^ndr}{\int_{0}^{\infty} \varphi^2\sinh^{n-2}dr}=\inf_{\varphi\in C^{\infty}_{cs}((0,\infty))}\frac{\int_{0}^{\infty} (\partial_s\varphi)^2(1+s^2)^{1/2}s^nds}{\int_{0}^{\infty} \varphi^2(1+s^2)^{-1/2}s^{n-2}ds},
\end{align}
which yields the lower estimate
\begin{align}
A\geq \inf_{\varphi\in C^{\infty}_{cs}((0,\infty))}\frac{\int_{0}^{\infty} (\partial_s\varphi)^2s^nds}{\int_{0}^{\infty} \varphi^2s^{n-2}ds}=\frac{(n-1)^2}{4}.
\end{align}
On the other hand we have $(1+s^2)^{1/2}\leq 1+s$ and $(1+s^2)^{-1/2}=1+s^2F(s)$ where $F(s)$ is smooth and bounded for all $s\geq0$. Let $\varphi_k$ be a sequence of compactly supported smooth functions such that
\begin{align}
\frac{\int_{0}^{\infty} (\partial_s\varphi_k)^2s^nds}{\int_{0}^{\infty} \varphi_k^2s^{n-2}ds}\overset{k\to\infty}{\to} \frac{(n-1)^2}{4}.
\end{align}
Because the quotient is invariant under rescaling $s\mapsto \alpha s$, we may assume that $\supp(\varphi_k)\subset (0,\epsilon_k)$ with $\epsilon_k\to 0$.
Thus,
\begin{align}
\frac{\int_{0}^{\infty}  (\partial_s\varphi_k)^2(1+s^2)^{1/2}s^nds}{\int_{0}^{\infty}  (\varphi_k)^2(1+s^2)^{-1/2}s^{n-2}ds}\leq \frac{\int_{0}^{\epsilon_k}  (\partial_s\varphi_k)^2(1+s)s^nds}{\int_{0}^{\epsilon_k} \varphi_k^2(1+s^2\cdot F(s))s^{n-2}ds}\leq \frac{(n-1)^2}{4}+\delta_k
\end{align}
and $\delta_k\to 0$ as $k\to\infty$. This proves that the first infiumum in the statement equals $1/4\cdot(n-1)^2$. Now we consider the other one. At first, it is clear that
\begin{align}
\inf_{\varphi\in C^{\infty}_{cs}((0,\infty))}\frac{\int_{0}^{\infty} (\partial_r\varphi)^2\sinh^ndr}{\int_{0}^{\infty} \varphi^2\cosh^2\sinh^{n-2}dr}\leq \inf_{\varphi\in C^{\infty}_{cs}((0,\infty))}\frac{\int_{0}^{\infty} (\partial_r\varphi)^2\sinh^ndr}{\int_{0}^{\infty} \varphi^2\sinh^{n-2}dr}&=\frac{(n-1)^2}{4}.
\end{align}
To prove the converse inequality, we first substitute $s=\sinh(r)$ which yields
\begin{align}
\inf_{\varphi\in C^{\infty}_{cs}((0,\infty))}\frac{\int_{0}^{\infty} (\partial_r\varphi)^2\sinh^ndr}{\int_{0}^{\infty} \varphi^2\cosh^2\sinh^{n-2}dr}=\inf_{\varphi\in C^{\infty}_{cs}((0,\infty))}\frac{\int_{0}^{\infty} (\partial_s\varphi)^2(1+s^2)^{1/2}s^nds}{\int_{0}^{\infty} \varphi^2(1+s^2)^{1/2}s^{n-2}ds}.
\end{align}
Now we substitute $\psi(s)=\varphi(s)\cdot (1+s^2)^{1/4}$. Then by straightforward calculations,
\begin{align}
\int_{0}^{\infty} \varphi^2(1+s^2)^{1/2}s^{n-2}ds=\int_{0}^{\infty} \psi^2s^{n-2}ds
\end{align}
and
\begin{equation}
\begin{aligned}
\int_{0}^{\infty} (\varphi')^2(1+s^2)^{1/2}s^nds=&
\int_{0}^{\infty} (\psi')^2s^nds+\frac{1}{4}\int_{0}^{\infty}  \psi^2 (1+s^2)^{-2}s^{n+2}ds\\&
-\int_{0}^{\infty} \psi'\psi(1+s^2)^{-1}s^n ds\\
=&\int_{0}^{\infty} (\psi')^2s^nds+\frac{1}{4}\int_{0}^{\infty} \psi^2(1+s^2)^{-2}s^{n+2}ds\\&
+\frac{1}{2}\int_{0}^{\infty} \psi^2(1+s^2)^{-2}[(n-2)s^{n+1}+ns^{n-1}] ds\\
\geq &\int_{0}^{\infty} (\psi')^2s^nds,
\end{aligned}
\end{equation}
where we used integration by parts for the second equality. Thus,
\begin{align}
\inf_{\varphi\in C^{\infty}_{cs}((0,\infty))}\frac{\int_{0}^{\infty} (\varphi')^2(1+s^2)^{1/2}s^nds}{\int_{0}^{\infty} \varphi^2(1+s^2)^{1/2}s^{n-2}ds}\geq \inf_{\psi\in C^{\infty}_{cs}((0,\infty))}\frac{\int_{0}^{\infty} (\psi')^2s^nds}{\int_{0}^{\infty} \psi^2s^{n-2}ds}=\frac{(n-1)^2}{4},
\end{align}
which proves the second formula.
\end{proof}
From now on, we can proceed similarly as in the previous subsections.
 We study the Einstein operator as a quadratic form on the subspaces $V_{k,i}$. Let
 \begin{align}
 \tilde{h}=\varphi r^2h_i\in V_{1,i}.
 \end{align}
 Then by Lemma \ref{infima_2},
\begin{equation}
\begin{aligned}
 (\tilde{\Delta}_E\tilde{h},\tilde{h})_{L^2(\tilde{g})}&=\int_{0}^{\infty}(\varphi')^2\sinh^n dr
 +\kappa_i\int_{0}^{\infty}\varphi^2\sinh^{n-2}dr\\
 &\geq\frac{(n-1)^2}{4}\int_{0}^{\infty}\varphi^2\cosh^2\sinh^{n-2}dr+\kappa_i\int_{0}^{\infty}\varphi^2\sinh^{n-2}dr\\ 
 &=\frac{(n-1)^2}{4}\left\|\tilde{h}\right\|^2_{L^2(\tilde{g})}+\left[\frac{(n-1)^2}{4}+\kappa_i\right]\int_{0}^{\infty}\varphi^2\sinh^{n-2}dr .
\end{aligned}
\end{equation}
 and thus, the quadratic form is positive definite if $\kappa_i\geq -\frac{1}{4}(n-1)^2$. On the other hand, since the inequality 
 \begin{align}
 (\tilde{\Delta}_E\tilde{h},\tilde{h})_{L^2(\tilde{g})}&\geq\frac{(n-1)^2}{4}\int_{0}^{\infty}\varphi^2\sinh^n dr
 +\kappa_i\int_{0}^{\infty}\varphi^2\sinh^{n-2}dr
 \end{align}
 is optimal, the scalar product can be negative for an appropriate $\varphi$ if $\kappa_i< -\frac{1}{4}(n-1)^2$. For
 \begin{align}
 \tilde{h}=\varphi v_i\tilde{g}\in V_{2,i},
 \end{align}
 we have
\begin{equation}
\begin{aligned}
 (\tilde{\Delta}_E\tilde{h},\tilde{h})_{L^2(\tilde{g})}&=
(n+1)\int_{0}^{\infty}(\varphi')^2\sinh^{n}dr+
(n+1)\lambda_i\int_{0}^{\infty}\varphi^2\sinh^{n-2}dr\\&\qquad+2n(n+1)\int_0^{\infty}\varphi^2\sinh^ndr\geq2n\left\|\tilde{h}\right\|^2_{L^2(\tilde{g})}.
\end{aligned}
\end{equation}
 Next, pick
\begin{align}
 \tilde{h}= \tilde{h}_1+ \tilde{h}_2=\varphi f^2\delta^*\omega_i+
\psi \cdot dr\odot f\omega_i \in V_{3,i}.
\end{align}
Then we have the scalar products
\begin{equation}
\begin{aligned}
 (\tilde{\Delta}_E\tilde{h}_1,\tilde{h}_1)_{L^2(\tilde{g})}&=
 \frac{1}{2}(\mu_i-(n-1))\int_{0}^{\infty}(\varphi')^2\sinh^{n} dr\\
 &\qquad+ \frac{1}{2}(\mu_i-(n-1))^2\int_{0}^{\infty}\varphi^2\sinh^{n-2}dr,\\
(\tilde{\Delta}_E\tilde{h}_2,\tilde{h}_2)_{L^2(\tilde{g})}&=2\mu_i\int_{0}^{\infty}\psi^2\sinh^{n-2}dr
 +(2n+6)\int_{0}^{\infty}\psi^2\cosh^2\sinh^{n-2}dr\\
&\qquad +2\int_{0}^{\infty}(\psi')^2\sinh^{n}dr-4\int_0^{\infty}\psi^2sinh^ndr,\\
 (\tilde{\Delta}_E\tilde{h}_1,\tilde{h}_2)_{L^2(\tilde{g})}&=-2(\mu_i-(n-1))\int_{0}^{\infty}\varphi\psi \cosh\sinh^{n-2}dr
\end{aligned}
\end{equation}
 and the estimates
\begin{equation}
\begin{aligned}
  (\tilde{\Delta}_E\tilde{h}_1,\tilde{h}_1)_{L^2(\tilde{g})}&\geq
 \frac{(n-1)^2}{8}(\mu_i-(n-1))\int_{0}^{\infty}\varphi^2\cosh^2\sinh^{n-2} dr\\
& \qquad+ \frac{1}{2}(\mu_i-(n-1))^2\int_{0}^{\infty}\varphi^2\sinh^{n-2}dr,\\
(\tilde{\Delta}_E\tilde{h}_2,\tilde{h}_2)_{L^2(\tilde{g})}&\geq 2\mu_i\int_{0}^{\infty}\psi^2\sinh^{n-2}dr
 +2(n+1)\int_{0}^{\infty}\psi^2\cosh^2\sinh^{n-2}dr\\ 
 &\qquad+4\int_{0}^{\infty}\psi^2\sinh^{n-2}dr+\frac{(n-1)^2}{2}\int_{0}^{\infty}\psi^2\cosh^2\sinh^{n-2}dr,\\
 2| (\tilde{\Delta}_E\tilde{h}_1,\tilde{h}_2)_{L^2(\tilde{g})} |&\leq \frac{1}{2}(\mu_i-(n-1))^2\int_{0}^{\infty}\varphi^2\sinh^{n-2}dr\\ 
 &\qquad+8\int_{0}^{\infty}\psi^2\cosh^2\sinh^{n-2}dr,
\end{aligned}
\end{equation}
 which imply that
\begin{equation}
\begin{aligned}
  (\tilde{\Delta}_E(\tilde{h}_1+\tilde{h}_2),\tilde{h}_1+\tilde{h}_2)_{L^2(\tilde{g})}&\geq \frac{(n-1)^2}{8}(\mu_i-(n-1))\int_{0}^{\infty}\varphi^2\cosh^2\sinh^{n-2} dr\\ &\qquad +\frac{(n-1)^2}{2}\int_{0}^{\infty}\psi^2\cosh^2\sinh^{n-2}dr\\
  &\geq \frac{(n-1)^2}{4}\left(\left\|\tilde{h}_1\right\|^2_{L^2(\tilde{g})}+\left\|\tilde{h}_2\right\|^2_{L^2(\tilde{g})}\right)=\frac{(n-1)^2}{4}\left\|\tilde{h}\right\|^2_{L^2(\tilde{g})}.
\end{aligned}
\end{equation}
   It now just remains to consider the spaces $V_{4,i}$. Let
 \begin{align}
  \tilde{h}= \tilde{h}_1+ \tilde{h}_2+\tilde{h}_3=\varphi f^2(n\nabla^2 v_i+\Delta v_i\cdot g)+\psi \cdot dr\odot \nabla v_i+\chi\cdot v_i(f^2g-ndr\otimes dr)\in V_{4,i}
 \end{align}
We have the scalar products 
\begin{equation}
\begin{aligned}
    (\tilde{\Delta}_E\tilde{h}_1,\tilde{h}_1)_{L^2(\tilde{g})}&= (n-1)n\lambda_i(\lambda_i-n)\int_0^{\infty}(\varphi')^2\sinh^n dr \\
 &\qquad+n(n-1)\lambda_i(\lambda_i-n)(\lambda_i-2(n-1))\int_0^{\infty}\varphi^2\sinh^{n-2}dr,\\
  (\tilde{\Delta}_E\tilde{h}_2,\tilde{h}_2)_{L^2(\tilde{g})}&=(2n+6)\lambda_i\int_{0}^{\infty} \psi^2\cosh^2\sinh^{n-2}dr+2\lambda_i\int_{0}^{\infty} (\psi')^2\sinh^ndr\\&\qquad+2\lambda_i\left(\lambda_i-(n-1)\right)\int_{0}^{\infty} \psi^2\sinh^{n-2}dr
  -4\lambda_i\int_0^{\infty}\psi^2\sinh^ndr,
\\
  (\tilde{\Delta}_E\tilde{h}_3,\tilde{h}_3)_{L^2(\tilde{g})}&= 
  n((n+1)\lambda_i-2(n-1))\int_{0}^{\infty} \chi^2 \sinh^{n-2}dr \\
  &\qquad
  +(n+1)n\int_{0}^{\infty}(\chi')^2\sinh^ndr
  +2n^2(n+3)\int_{0}^{\infty} \chi^2\cosh^2\sinh^{n-2}dr\\
  &\qquad-4n^2\int_0^{\infty}\varphi^2\sinh^ndr
\end{aligned}
\end{equation}
  and
\begin{equation}
\begin{aligned}
(\tilde{\Delta}_E\tilde{h}_1,\tilde{h}_2)_{L^2(\tilde{g})}&= 
-4(n-1)\lambda_i(\lambda_i-n)\int_0^{\infty}\varphi\psi\cosh \sinh^{n-2}dr,\\
(\tilde{\Delta}_E\tilde{h}_2,\tilde{h}_3)_{L^2(\tilde{g})}&= 
4(n+1)\lambda_i\int_0^{\infty}\psi\chi \cosh \sinh^{n-2}dr.
\end{aligned}
\end{equation} 
   By Lemma \ref{infima_2}, we have lower estimates
\begin{equation}
\begin{aligned}    (\tilde{\Delta}_E\tilde{h}_1,\tilde{h}_1)_{L^2(\tilde{g})}&\geq (n-1)n\lambda_i(\lambda_i-n)^2\int_0^{\infty}\varphi^2\sinh^{n-2} dr \\
 &\qquad+\frac{1}{4}n(n-1)(n-3)^2\lambda_i(\lambda_i-n)\int_0^{\infty}\varphi^2\cosh^2\sinh^{n-2}dr\\&\qquad+n(n-1)(n-2)\lambda_i(\lambda_i-n)\int_0^{\infty}\varphi^2\sinh^{n}dr,\\
   (\tilde{\Delta}_E\tilde{h}_2,\tilde{h}_2)_{L^2(\tilde{g})}&\geq(2n+2+\frac{1}{2}(n-1)^2)\lambda_i\int_0^{\infty} \psi^2\cosh^2\sinh^{n-2}dr\\
   &\qquad+6\lambda_i\int_0^{\infty}\psi^2\sinh^{n-2}dr
   +2\lambda_i(\lambda_i-n)\int_0^{\infty} \psi^2\sinh^{n-2}dr,\\
     (\tilde{\Delta}_E\tilde{h}_3,\tilde{h}_3)_{L^2(\tilde{g})}&\geq
  n(n+1)\lambda_i\int_0^{\infty} \chi^2 \sinh^{n-2}dr
  +2n(n+1)\int_0^{\infty} \chi^2 \sinh^{n-2}dr\\
  &\qquad+\left[\frac{1}{4}(n+1)n(n-1)^2+2n^2(n+1)\right]\int_0^{\infty}\chi^2\cosh^2\sinh^{n-2}dr,
\end{aligned}
\end{equation}  
and for the off-diagonal terms, we use the Young inequality to show
\begin{equation}
\begin{aligned}
  2|(\tilde{\Delta}_E\tilde{h}_1,\tilde{h}_2)_{L^2(\tilde{g})}|&\leq
  2\alpha (n-1)\lambda_i(\lambda_i-n)\int_0^{\infty}|\varphi\psi| \cosh\sinh^{n-2}dr\\&\qquad+
  2(4-\alpha)(n-1)\lambda_i(\lambda_i-n)\int_0^{\infty}|\varphi\psi| \cosh\sinh^{n-2}dr\\
  &\leq n(n-1)\lambda_i(\lambda_i-n)^2\int_0^{\infty}\varphi^2\sinh^{n-2}dr\\
&\qquad  + \alpha^2\frac{n-1}{n}\lambda_i\int_0^{\infty}\psi^2\cosh^2\sinh^{n-2}dr\\
  &\qquad+\frac{1}{4}n(n-1)(n-3)^2\lambda_i(\lambda_i-n)\int_0^{\infty}\varphi^2\cosh^2\sinh^{n-2}dr\\
  &\qquad+ (4-\alpha)^2\frac{4(n-1)}{n(n-3)^2}\lambda_i(\lambda_i-n)\int_0^{\infty}\psi^2\sinh^{n-2}dr,\\
   2|(\tilde{\Delta}_E\tilde{h}_2,\tilde{h}_3)_{L^2(\tilde{g})}|&\leq 
   16\frac{n+1}{n}\lambda_i \int_0^{\infty}\psi^2\cosh^2\sinh^{n-2}dr\\
& \qquad  +n(n+1)\lambda_i\int_0^{\infty}\chi^2\sinh^{n-2}dr,
\end{aligned}
\end{equation}
where $\alpha\in[0,4]$ is some parameter which we choose later. We get
\begin{equation}
\begin{aligned}
(\tilde{\Delta}_E&(\tilde{h}_1+\tilde{h}_2+\tilde{h}_3),\tilde{h}_1+\tilde{h}_2+\tilde{h}_3)_{L^2(\tilde{g})}\\ &\geq
n(n-1)(n-2)\lambda_i(\lambda_i-n)\int_0^{\infty}\varphi^2\sinh^{n}dr\\ &\qquad+
 \left[2n+2+\frac{1}{2}(n-1)^2-16\frac{n+1}{n}-\alpha^2\frac{n-1}{n}\right]\lambda_i\int_0^{\infty} \psi^2\cosh^2\sinh^{n-2}dr\\
&\qquad+\left[2-(4-\alpha)^2\frac{4(n-1)}{n(n-3)^2}\right]\lambda_i(\lambda_i-n)\int_0^{\infty} \psi^2\cosh^2\sinh^{n-2}dr\\
&\qquad+\left[\frac{1}{4}(n+1)n(n-1)^2+2n^2(n+1)\right]\int_0^{\infty}\chi^2\cosh^2\sinh^{n-2}dr,
\end{aligned}
\end{equation} 
and if $n\geq 6$ we can choose $\alpha=2$ to get the second and the third term of the right hand side positive. In this case, we therefore get
\begin{equation}
\begin{aligned}
(\tilde{\Delta}_E(\tilde{h}_1+\tilde{h}_2+\tilde{h}_3),\tilde{h}_1+\tilde{h}_2+\tilde{h}_3)_{L^2(\tilde{g})}&\geq C(n)\left(\left\|\tilde{h}_1\right\|^2_{L^2(\tilde{g})}+\left\|\tilde{h}_2\right\|^2_{L^2(\tilde{g})}+\left\|\tilde{h}_3\right\|^2_{L^2(\tilde{g})}\right)\\&=C(n)\left\|\tilde{h}\right\|^2_{L^2(\tilde{g})}.
\end{aligned}
\end{equation} 
For $n=4$ and $n=5$, we cannot choose such an $\alpha$, so we have to treat these cases separately. Let us consider the five-dimensional case first. We then have the estimates
\begin{equation}
\begin{aligned}    (\tilde{\Delta}_E\tilde{h}_1,\tilde{h}_1)_{L^2(\tilde{g})}&\geq 20\lambda_i(\lambda_i-5)^2\int_0^{\infty}\varphi^2\sinh^{3} dr \\
 &\qquad+20\lambda_i(\lambda_i-5)\int_0^{\infty}\varphi^2\cosh^2\sinh^{3}dr,\\
   (\tilde{\Delta}_E\tilde{h}_2,\tilde{h}_2)_{L^2(\tilde{g})}&\geq20\lambda_i\int_0^{\infty} \psi^2\cosh^2\sinh^{3}dr+2\lambda_i(\lambda_i-2)\int_0^{\infty} \psi^2\sinh^{3}dr,\\
     (\tilde{\Delta}_E\tilde{h}_3,\tilde{h}_3)_{L^2(\tilde{g})}&\geq
  (30\lambda_i+60)\int_0^{\infty} \chi^2 \sinh^{3}dr+396\int_0^{\infty}\chi^2\cosh^2\sinh^{3}dr.
\end{aligned}
\end{equation} 
  We have two terms on the right hand sides, where we denote the coefficients in front of the first integrals by $a_{ii}$ and the coefficients in front of the second integrals by $b_{ii}$, $i=1,2,3$.
  The off-diagonal terms are
\begin{equation}
\begin{aligned}
(\tilde{\Delta}_E\tilde{h}_1,\tilde{h}_2)_{L^2(\tilde{g})}&= 
-16\lambda_i(\lambda_i-5)\int_0^{\infty}\varphi\psi\cosh \sinh^{3}dr,\\
(\tilde{\Delta}_E\tilde{h}_2,\tilde{h}_3)_{L^2(\tilde{g})}&= 
24\lambda_i\int_0^{\infty}\psi\chi \cosh \sinh^{3}dr.
\end{aligned}
\end{equation} 
  Let $a_{12}=a_{21}=-16\lambda_i(\lambda_i-5)$ and $a_{23}=a_{32}=b_{23}=b_{32}=12\lambda_i$. Then $\tilde{\Delta}_E$ on $V_{4,i}$ can be estimated from below by the sum of two quadratic forms, given by the matrices $A=(a_{ij})_{1\leq i,j\leq 3}$ and $B=(b_{ij})_{1\leq i,j\leq 3}$. Assume at first that $\varphi=\chi=\cosh(r)\cdot\psi$ and that $\int_0^{\infty}\varphi^2\sinh^3dr=1$. On the space
  \begin{align}
 \text{span}\left\{\varphi f^2(n\nabla^2 v_i+\Delta v_i\cdot g),\cosh\varphi \cdot dr\odot \nabla v_i,\varphi\cdot v_i(f^2g-ndr\otimes dr)\right\},
\end{align}
$A$ can be represented as
\begin{align}
\begin{pmatrix}20\lambda_i(\lambda_i-5)^2 & -16\lambda_i(\lambda_i-5) & 0\\
-16\lambda_i(\lambda_i-5) & 20\lambda_i & 12\lambda_i \\
0 & 12\lambda_i & 30\lambda_i+60 
\end{pmatrix},
\end{align}
which is positive definite since $\lambda_i\geq5$ for $i\geq 2$. The case $\lambda_1=0$ is clear. For $\varphi,\psi,\chi\in C^{\infty}_{cs}((0,\infty))$ arbitrarily, we can use a development to an orthonormal basis of an inner product which is given by $\langle \varphi,\psi\rangle =\int_0^{\infty}\varphi\cdot \psi\sinh^3dr$ to show that $A$ is positive definite on all of $V_{4,i}$. Simlarly, we treat $B$. Choose $\psi=\cosh(r)\cdot\varphi=\cosh(r)\cdot\chi$ such that $\int_0^{\infty}\psi^2\sinh^3dr=1$. On
\begin{align}
 \text{span}\left\{\cosh\varphi f^2(n\nabla^2 v_i+\Delta v_i\cdot g),\varphi \cdot dr\odot \nabla v_i,\cosh\varphi\cdot v_i(f^2g-ndr\otimes dr)\right\},
\end{align}
$B$ is given by
\begin{align}
\begin{pmatrix}20\lambda_i(\lambda_i-5) & 0 & 0\\
0 & 2\lambda_i(\lambda_i-2) & 12\lambda_i \\
0 & 12\lambda_i & 396 
\end{pmatrix},
\end{align}
which is also positive definite and as above, this shows that $B$ is positive definite on all of $V_{4,i}$. By making the diagonal elements of $A,B$ slightly smaller (which does not affect definiteness), we also get an estimate 
\begin{align}
(\tilde{\Delta}_E(\tilde{h}_1+\tilde{h}_2+\tilde{h}_3),\tilde{h}_1+\tilde{h}_2+\tilde{h}_3)_{L^2(\tilde{g})}&\geq C(5)\left\|\tilde{h}\right\|^2_{L^2(\tilde{g})}.
\end{align}
The four-dimensional case is treated analogously. In this case, we have lower estimates
\begin{equation}
\begin{aligned}   (\tilde{\Delta}_E\tilde{h}_1,\tilde{h}_1)_{L^2(\tilde{g})}&\geq 12\lambda_i(\lambda_i-4)^2\int_0^{\infty}\varphi^2\sinh^{2} dr \\
 &\qquad+3\lambda_i(\lambda_i-4)\int_0^{\infty}\varphi^2\cosh^2\sinh^{2}dr,\\
   (\tilde{\Delta}_E\tilde{h}_2,\tilde{h}_2)_{L^2(\tilde{g})}&\geq\frac{29}{2}\lambda_i\int_0^{\infty} \psi^2\cosh^2\sinh^{2}dr+2\lambda_i(\lambda_i-1)\int_0^{\infty} \psi^2\sinh^{2}dr,\\
     (\tilde{\Delta}_E\tilde{h}_3,\tilde{h}_3)_{L^2(\tilde{g})}&\geq
  (20\lambda_i+40)\int_0^{\infty} \chi^2 \sinh^{2}dr+205\int_0^{\infty}\chi^2\cosh^2\sinh^{2}dr
\end{aligned}
\end{equation} 
  and the off-diagonal terms are
\begin{equation}
\begin{aligned}
(\tilde{\Delta}_E\tilde{h}_1,\tilde{h}_2)_{L^2(\tilde{g})}&= 
-12\lambda_i(\lambda_i-4)\int_0^{\infty}\varphi\psi\cosh \sinh^{2}dr,\\
(\tilde{\Delta}_E\tilde{h}_2,\tilde{h}_3)_{L^2(\tilde{g})}&= 
20\lambda_i\int_0^{\infty}\psi\chi \cosh \sinh^{2}dr.
\end{aligned}
\end{equation} 
  One can estimate $\tilde{\Delta}_E$ from below by the sum of two quadratic forms $A$ and $B$, which can be represented by the matrices
  \begin{align}
 \begin{pmatrix}12\lambda_i(\lambda_i-4)^2 & -12\lambda_i(\lambda_i-4) & 0\\
-12\lambda_i(\lambda_i-4) & \frac{29}{2}\lambda_i & 5\lambda_i \\
0 & 5\lambda_i & 20\lambda_i+40
\end{pmatrix} ,
\quad
\begin{pmatrix}3\lambda_i(\lambda_i-4) & 0 & 0\\
0 & 2\lambda_i(\lambda_i-1) & 15\lambda_i \\
0 & 15\lambda_i & 205 
\end{pmatrix}.
  \end{align}
  They are both positive definite for $i\geq2$ since $\lambda_i\geq 4$. The case $\lambda_i=0$ is again clear.
By the same arguments as above we also get for $n=4$ the estimate
  \begin{align}
(\tilde{\Delta}_E(\tilde{h}_1+\tilde{h}_2+\tilde{h}_3),\tilde{h}_1+\tilde{h}_2+\tilde{h}_3)_{L^2(\tilde{g})}&\geq C(4)\left\|\tilde{h}\right\|^2_{L^2(\tilde{g})}.
\end{align}
\section{Examples}\label{section4}
In this section, we want to apply the above theorems do deduce stability or instability for a large class of warped product Einstein manifolds.

\begin{thm}
Let $(\widetilde{M},\tilde{g})$ be a complete Riemannian spin manifold carrying an imaginary Killing spinor. Then it is a strictly stable Einstein manifold of negative scalar curvature. 
\end{thm}
\begin{proof}Recall that a spinor $\sigma$ is called an imaginary Killing spinor if there exists a $\lambda\in \R$ such that $\nabla_X\sigma=i\lambda X\cdot \sigma$ for any vector field $X$. The constant $\lambda$ is called the Killing number. It is well known that such a spinor forces the underlying metric to be Einstein with scalar curvature $-4\lambda^2n(n-1)$. By a classification result of Helga Baum \cite{Bau89}, any complete manifold carrying an imaginary Killing spinor with Killing number $\lambda$ is of the form
\begin{align}
(\widetilde{M},\tilde{g})=(\R\times M,dr^2+e^{4\lambda r}g),
\end{align}
where $(M,g)$ is a Ricci-flat manifold carrying a parallel spinor. Due to rescaling, we may assume that $\lambda=1/2$. Since $(M,g)$ is stable due to \cite[Theorem 1.1]{DWW05}, Theorem \ref{thmexponentialcone} implies that $(\widetilde{M},\tilde{g})$ is strictly stable.
\end{proof}
\begin{rem}
In the case of Ricci flat or positive Einstein manifolds, the presence of parallel resp.\ Killing spinors yields Bochner formulas which imply lower bounds on the eigenvalues of the Einstein operator, see \cite{Wan91,GHP03,DWW05}. Unfortunately, this method can not be used in the case of imaginary Killing spinors, because in contrast to the other cases, their pointwise length is not constant. 
\end{rem}
\noindent
In the remainder of this section we apply Theorem \ref{thmricciflatcone} and Theorem \ref{thmhyperboliccone} in various cases.
\begin{prop}Let $(M,g)$ be Einstein with scalar curvature $n(n-1)$ and suppose that $(M,g)$ is K\"ahler or has nonnegative sectional curvature. Then the smallest eigenvalue of the Einstein operator satisfies the bound 
\begin{align}
\lambda\geq -2(n-1).
\end{align}
\end{prop}
\begin{proof}
In the K\"ahler case, the bound follows essentially from the calculations in \cite{Koi83}, see also \cite[pp. 362--363]{Bes08}. In the case of nonnegative sectional curvature, the quadratic form 
\begin{align*}h\mapsto (h,\ric\circ h+h\circ\ric-2\mathring{R}h)_{L^2}
 \end{align*}
is nonnegative by \cite[Lemma 2.4]{Bar93}. Then the eigenvalue bound follows since $\ric_g=(n-1)g$.
\end{proof}
\noindent
As a consequence, we get
\begin{thm}
Let $(M^n,g)$, $n\geq9$ be an Einstein manifold of scalar curvature $n(n-1)$ which is K\"ahler or has nonnegative sectional curvature. Then the Ricci-flat cone over $(M,g)$ is stable.
\end{thm}
\begin{thm}
Let $(M^n,g)$, $n\geq9$ be an Einstein manifold of scalar curvature $n(n-1)$ which is K\"ahler or has nonnegative sectional curvature. Then the hyperbolic cone over $(M,g)$ is strictly stable.
\end{thm}
On the other hand, as was pointed out in \cite{HHS14}, $-2(n-1)\in \spectrum(\Delta_E|_{TT})$ if $(M,g)$ is a product of positive Einstein manifolds or if $(M,g)$ is a positive K\"ahler-Einstein manifolds with $\dimn(H_{1,1}(M))>1$. Thus, we obtain
\begin{thm}[{\cite[Theorem 1.1 and Theorem 1.2]{HHS14}}]
Let $(M^n,g)$, $4\leq n\leq8$ be an Einstein manifold of scalar curvature $n(n-1)$ which is K\"ahler with $\dimn(H^{1,1}(M))>1$ or is a product of Einstein manifolds. Then the Ricci-flat cone over $(M,g)$ is unstable.
\end{thm}
\begin{thm}
Let $(M^n,g)$, $4\leq n\leq8$ be an Einstein manifold of scalar curvature $n(n-1)$ which is K\"ahler with $\dimn(H^{1,1}(M))>1$ or is a product of Einstein manifolds. Then the hyperbolic cone over $(M,g)$ is unstable.
\end{thm}
Suppose that $(M,g)$ Einstein with scalar curvature $n(n-1)$ and carries a real Killing spinor. Then all eigenvalues of $\Delta_E|_{TT}$ have the lower bound $-\frac{1}{4}(n-1)^2$, or equivalently, all eigenvalues of the Lichnerowicz Laplacian on $TT$-tensors have the lower bound $4-\frac{(n-5)^2}{4}$ \cite{GHP03}. Thus, the Ricci flat cone over it is stable, but this is already clear since the cone carries a parallel spinor. The additional information that we get is
\begin{thm}
Let $(M,g)$ be an Einstein manifold of scalar curvature $n(n-1)$ which carries a real Killing spinor. Then the hyperbolic cone over $(M,g)$ is strictly stable.
\end{thm}
Another class of Einstein manifolds satisfying the bound $\Delta_E|_{TT}\geq-\frac{1}{4}(n-1)^2$ are the symmetric spaces of compact type. Since symmetric spaces of compact type have nonnegative sectional curvature, this bound has to be checked just in dimensions $4\leq n\leq 8$. For this cases, it suffices to consider the tables in \cite[pp.\ 243-245]{CH13}, where the smallest eigenvalues of the Lichnerowicz Laplacian on those spaces is collected.
\begin{thm}
The Ricci-flat cone and the hyperbolic cone over any symmetric space of compact type are stable resp.\ strictly stable.
\end{thm}
Many more positive Einstein manifolds violating the bound $\Delta_E|_{TT}\geq-\frac{1}{4}(n-1)^2$ are known. The reader is referred to \cite{PP84a,PP84b,GH02,GHP03,HHS14} for further examples.


\begin{thebibliography}{DWW07}


\providecommand{\url}[1]{\texttt{#1}}
\expandafter\ifx\csname urlstyle\endcsname\relax
  \providecommand{\doi}[1]{doi: #1}\else
  \providecommand{\doi}{doi: \begingroup \urlstyle{rm}\Url}\fi

\bibitem[AM11]{AMo11}
\textsc{Andersson}, Lars ; \textsc{Moncrief}, Vincent:
\newblock {Einstein spaces as attractors for the Einstein flow.}
\newblock {In: }\emph{J. Differ. Geom.} \textbf{89} (2011), no. 1, 1--47

\bibitem[Bam15]{Bam15}
\textsc{Bamler}, Richard:
\newblock{Stability of symmetric spaces of noncompact type under Ricci flow.}
\newblock {In: }\emph{Geom. Funct. Anal.} \textbf{25} (2015), no 2, 342--416

\bibitem[Bar93]{Bar93}
\textsc{Barmettler}, Urs:
\newblock \emph{On the {L}ichnerowicz {L}aplacian.} PhD thesis, ETH Z\"urich, 1993

\bibitem[Bau89]{Bau89}
\textsc{Baum}, Helga:
\newblock {Complete Riemannian manifolds with imaginary Killing spinors.}
\newblock {In: Ann. Glob. Anal. Geom.} \textbf{7} (1989), no. 3, 205--226

\bibitem[Ber65]{Ber65}
\textsc{Berger}, Marcel:
\newblock {Sur les vari\'et\'es d'Einstein compactes.}
\newblock {In: }\emph{Comptes {R}endus de la {III}e {R}\'eunion du {G}roupement
  des {M}ath\'ematiciens d'{E}xpression {L}atine}.
\newblock 1965, 35--55

\bibitem[Bes08]{Bes08}
\textsc{Besse}, Arthur~L.:
\newblock \emph{{Einstein manifolds. Reprint of the 1987 edition.}}
\newblock {Berlin: Springer}, 2008
%
%
%
%

\bibitem[CHI04]{CHI04}
\textsc{Cao}, Huai-Song ; \textsc{Hamilton}, Richard  ; \textsc{Ilmanen}, Tom:
\newblock Gaussian densities and stability for some {R}icci solitons.
\newblock   Preprint,
\newblock arXiv:math/0404165


\bibitem[CH15]{CH13}
\textsc{Cao}, Huai-Dong ; \textsc{He}, Chenxu:
\newblock Linear {S}tability of {P}erelmans $\nu$-entropy on {S}ymmetric spaces
  of compact type.
\newblock  In: \emph{J. Reine Angew. Math.} \textbf{709} (2015), 229--246.

%

%
\bibitem[Dai07]{Dai07}
\textsc{Dai}, Xianzhe:
\newblock Stability of {E}instein {M}etrics and {S}pin {S}tructures.
\newblock {In: }\emph{Proceedings of the 4th International Congress of Chinese
  Mathematicians} Vol II (2007), 59--72

\bibitem[DWW05]{DWW05}
\textsc{Dai}, Xianzhe ; \textsc{Wang}, Xiaodong  ; \textsc{Wei}, Guofang:
\newblock {On the stability of Riemannian manifold with parallel spinors.}
\newblock {In: }\emph{Invent. Math.} \textbf{161} (2005), no. 1, 151--176
%
%

\bibitem[FIK03]{FIK03}
\textsc{Feldman}, Mikhail ; \textsc{Ilmanen}, Tom ; \textsc{Knopf}, Dan:
\newblock Rotationally symmetric shrinking and expanding gradient K\"ahler-Ricci solitons.
\newblock {In: }\emph{J. Differ. Geom.} \textbf{65} (2003), no 2, 169--209

\bibitem[GIK02]{GIK02}
\textsc{Guenther}, Christine ; \textsc{Isenberg}, James ; \textsc{Knopf}, Dan:
\newblock {Stability of the Ricci flow at Ricci-flat metrics.}
\newblock {In: }\emph{Comm. Anal. Geom.} \textbf{10} (2002), no. 4, 741--777

\bibitem[GK04]{GK04}
\textsc{Gastel}, Andreas ; \textsc{Kronz}, Manfred:
\newblock A family of expanding Ricci solitons.
\newblock {In: }\emph{Variational Problems in Riemannian Geometry.} 2004, 81--93

\bibitem[GH02]{GH02}
\textsc{Gibbons}, Gary~W. ; \textsc{Hartnoll}, Sean~A.:
\newblock Gravitational instability in higher dimensions.
\newblock {In: }\emph{Phys. Rev. D.} \textbf{66} (2002), no. 6

\bibitem[GHP03]{GHP03}
\textsc{Gibbons}, Gary~W. ; \textsc{Hartnoll}, Sean~A.  ; \textsc{Pope},
  Christopher~N.:
\newblock Bohm and {E}instein-{S}asaki {M}etrics, {B}lack {H}oles, and
  {C}osmological {E}vent {H}orizons.
\newblock {In: }\emph{Phys. Rev. D} \textbf{67} (2003), no. 8

%

\bibitem[GPY82]{GPY82}
\textsc{Gross}, David~J. ; \textsc{Perry}, Malcolm~J.  ; \textsc{Yaffe},
  Laurence~G.:
\newblock Instability of flat space at finite temperature.
\newblock {In: }\emph{Phys. Rev. D} \textbf{25} (1982), no. 2, 330--355

\bibitem[HHS14]{HHS14}
\textsc{Hall}, Stuart ; \textsc{Haslhofer}, Robert; \textsc{Siepmann}, Michael:
\newblock The stability inequality for Ricci-flat cones.
\newblock {In: }\emph{J. Geom. Anal.} (2014), no. 1, 472--494.

\bibitem[HM14]{HM14}
 \textsc{Haslhofer}, Robert; \textsc{M\"{u}ller}, Reto:
\newblock Dynamical stability and instability of Ricci-flat metrics.
\newblock {In: }\emph{Math. Ann.} (2014), no. 1-2, 547--553.
%

\bibitem[Koi83]{Koi83}
\textsc{Koiso}, Norihito:
\newblock {Einstein metrics and complex structures.}
\newblock {In: }\emph{Invent. Math.} \textbf{73} (1983), 71--106

\bibitem[Kr\"o13]{Kro13}
\textsc{Kr\"oncke}, Klaus:
\newblock {Ricci flow, Einstein metrics and the Yamabe invariant.}
\newblock Preprint, arXiv:1312.2224

\bibitem[Kr\"o15a]{Kro15}
\textsc{Kr\"oncke}, Klaus:
\newblock {On infinitesimal Einstein deformations.}
\newblock {In: } \emph{Differ. Geom. Appl.} \textbf{38} (2015), 41--57 

\bibitem[Kr\"o15b]{Kro15b}
\textsc{Kr\"oncke}, Klaus:
\newblock {Stability and instability of Ricci solitons.}
\newblock {In: } \emph{ Calc. Var. Partial Differ. Equ.} \textbf{53} (2015), no. 1-2, 265--287 
%

%
%
%

\bibitem[Oba62]{Ob62}
\textsc{Obata}, Morio:
\newblock {Certain conditions for a Riemannian manifold to be isometric with a
  sphere.}
\newblock {In: }\emph{J. Math. Soc. Japan} \textbf{14} (1962), 333--340

\bibitem[PP84a]{PP84b}
\textsc{Page}, Don~N. ; \textsc{Pope}, Christopher~N.:
\newblock Stability analysis of compactifications of {D} = 11 supergravity with
  {SU}(3)$\times${SU}(2)$\times${U}(1) symmetry.
\newblock {In: }\emph{Phys. Lett.} \textbf{145} (1984), no. 5, 337--341

\bibitem[PP84b]{PP84a}
\textsc{Page}, Don~N. ; \textsc{Pope}, Christopher~N.:
\newblock Which compactifications of {D} = 11 supergravity are stable?
\newblock {In: }\emph{Phys. Lett.} \textbf{144} (1984), no. 5-6, 346--350

\bibitem[Ses06]{Ses06}
\textsc{Sesum}, Natasa:
\newblock {Linear and dynamical stability of Ricci-flat metrics.}
\newblock {In: }\emph{Duke Math. J.} \textbf{133} (2006), no. 1, 1--26

\bibitem[Sie13]{Sie13}
\textsc{Siepmann}, Michael:
\newblock \emph{Ricci Flows of Ricci flat Cones}
\newblock PhD thesis, ETH Z\"urich, 2013

\bibitem[SSS08]{SSS08}
\textsc{Schn\"urer}, Oliver; \textsc{Schulze}, Felix; \textsc{Simon}, Miles:
\newblock{Stability of Euclidean space under Ricci flow}
\newblock{In: }\emph{Comm. Anal. Geom.} \textbf{16} (2008), no. 1, 127--158.

\bibitem[SS13]{SS13}
\textsc{Schulze}, Felix; \textsc{Simon}, Miles:
\newblock{Expanding solitons with non-negative curvature operator coming out of cones.}
\newblock {In: }\emph{Math. Z.} \textbf{275} (2013), no. 1-2, 625--639.


%

\bibitem[Wan91]{Wan91}
\textsc{Wang}, McKenzie~Y.:
\newblock {Preserving parallel spinors under metric deformations.}
\newblock {In: }\emph{Indiana Univ. Math. J.} \textbf{40} (1991), no. 3, 815--844

\bibitem[War06]{War06}
\textsc{Warnick}, Claude:
\newblock {Semi-classical stability of AdS NUT instantons.}
\newblock {In: }\emph{Class. Quant. Grav.} \textbf{23} (2006), no. 11, 3801


%
%

\bibitem[Zhu11]{Zhu11}
\textsc{Zhu}, Meng:
\newblock {The second variation of the Ricci expander entropy}
\newblock {In: }\emph{Pac. J. Math.} \textbf{251} (2011), no. 2, 499--510

\end{thebibliography}
\end{document}